\documentclass[3p,times,twocolumn]{elsarticle}   

\usepackage[utf8]{inputenc}
\usepackage{graphicx}          
\graphicspath{ {Figures_bnd_wave/} }
\usepackage{amsmath}
\usepackage{natbib}
\usepackage{amsfonts}
\usepackage{color}
\usepackage{amssymb}
\usepackage{mathrsfs}
\usepackage{float}
\usepackage{enumerate}
\usepackage{verbatim}
\usepackage{setspace}
\usepackage{epsfig}
\usepackage{url}
\usepackage{graphicx}
\usepackage{lscape}
\usepackage{subcaption}
\usepackage{caption}
\usepackage{bbm}
\usepackage{arydshln}

\numberwithin{equation}{section}
\newtheorem{remark}{Remark}[section]

\newtheorem{theorem}{Theorem}[section]

\newtheorem{lemma}{Lemma}[section]

\newenvironment{proof}{{\textbf{Proof}:~}}{\hfill$\Box$\\}

\begin{document}

\begin{frontmatter}
\title{Global  finite-dimensional observer-based stabilization of a semilinear heat equation with large input delay\tnoteref{t1}}

\tnotetext[t1]{Supported by Israel Science Foundation (grant  673/19), the C. and H. Manderman Chair 
at Tel Aviv University and by the Y. and C. Weinstein Research
Institute for Signal Processing }
\author[Tel-Aviv university]{Rami Katz}\ead{rami@benis.co.il}
\author[Tel-Aviv university]{Emilia Fridman}\ead{emilia@eng.tau.ac.il}
\address[Tel-Aviv university]{School of Electrical Engineering, Tel-Aviv University, Tel-Aviv}  
%
%
\begin{keyword}                           
Nonlinear systems, distributed parameter systems, 
time-delay systems, observer-based control.           
\end{keyword}                             
%

\begin{abstract}
We study global  finite-dimensional observer-based  stabilization of a semilinear 1D heat equation with globally Lipschitz semilinearity in the state variable.
We consider Neumann actuation and point measurement.
Using dynamic extension and modal decomposition we derive nonlinear ODEs for the modes of the state.
We propose a controller that is based on a nonlinear finite-dimensional Luenberger observer. Our Lypunov $H^1$-stability analysis leads to LMIs, which are shown to be feasible for a large enough observer dimension and small enough Lipschitz constant. Next, we consider the case of  a constant input delay $r>0$. To compensate the delay, we introduce a chain of $M$ sub-predictors that leads to a nonlinear closed-loop ODE system, coupled with nonlinear infinite-dimensional tail ODEs. We provide LMIs for $H^1$-stability  and prove that for any $r>0$, the LMIs are feasible provided $M$ and $N$ are large enough and the Lipschitz constant is small enough. Numerical examples demonstrate the efficiency of the proposed approach.
\end{abstract}
\end{frontmatter}

\section{Introduction}
Observer-based control
of parabolic PDEs is a challenging problem with numerous applications, including  chemical reactors, flame propagation and viscous flow
\cite{christofides2001}.
Output-feedback controllers for  PDEs have been constructed by the modal decomposition approach \cite{curtain1982finite,lasiecka2000control,orlov2004robust}, the backstepping method \cite{Krstic2008} and the spatial decomposition approach \cite{Aut12,kang2020constrained}. Constructive finite-dimensional observer-based design for linear 1D parabolic PDEs was introduced in \cite{RamiContructiveFiniteDim,katz2021finiteKSE}, via modal decomposition.
 The challenging problem of efficient finite-dimensional observer-based design for semilinear parabolic PDEs remained open.

State-feedback control of some semilinear PDEs was studied in \cite{vazquez2008control1} using backstepping,  in \cite{karafyllis2019small} using small-gain theorem and in \cite{karafyllis2021lyapunov} via control Lyapunov functions.
Recently, modal-decomposition-based state-feedback was proposed in \cite{Nonl_Aut21} for global stabilization of heat equation and in  \cite{katz2021finiteKSE}
for regional stabilization of  Kuramoto-Sivashinsky equation.   Finite-dimensional control based on linear observers was proposed in \cite{wu2016finite} for semilinear parabolic PDEs via modal decomposition.
Linear observers should have high gains required to dominate the nonlinearity, which leads to small delays  that preserve the stability  \cite{lei2016highRami,najafi2021decreaseRami}.

For  ODEs, compensation of input delay can be achieved using three main predictor methods: the classical predictor 
\cite{Artstein82}, the PDE-based prtedictor \cite{Krstic09} and sequential sub-predictors (observers of the future state) \cite{
najafi2013closed1}. 
For  delay compensation of input/output delays in the case of
 nonlinear ODEs see e.g. \cite{Ahmed-ali2012,bekiaris2013nonlinear,bresch2015prediction,cacace2016stabilizationRami,germani2002new1,karafyllis2017predictor,mazenc2016stabilization} and the references therein).
 For semilinear heat equation, by using spatial decomposition, a chain of PDE observers (to compensate output delay) was suggested in \cite{ahmed2019observer}. For linear  heat equation,
 a classical state-feedback predictor via modal decomposition was
  proposed in \cite{prieur2018feedback}, whereas
  a sub-predictor 
  based on PDE observer  was suggested in \cite{selivanov2018delayed}.
For linear parabolic PDEs, finite-dimensional observe-based classical predictors and 
sub-predictors  were introduced  in \cite{katz2021sub}.

For semilinear parabolic PDEs, efficient finite-dimensional observer-based controller design as well as input delay compensation
remained  open challenging problems that we solve in the present paper.
We consider global stabilization of a semilinear heat equation under Neumann actuation and point measurement.
 The semilinarity is assumed to be globally Lipschitz in the state. Using dynamic extension and modal decomposition we derive nonlinear ODEs for the modes of the state.  
 We design a linear controller, which is based on finite-dimensional \emph{nonlinear} observer. 
 The challenge in the Lyapunov-based analysis is due to the coupling between the finite-dimensional and infinite-dimensional parts of the closed-loop system, introduced by both the semilinearity and the estimation error.
 Our  $H^1$-stability analysis leads to LMIs, which are shown to be feasible for a large enough observer dimension and small enough Lipschitz constant.

  We further consider the case of constant input delay $r>0$ and suggest compensating the delay using chain of $M$ sub-predictors - observers of the future state. 
 We
 introduce an approximate nonlinearity into the sub-predictor ODEs and provide $H^1$-stability analysis, where the difference between the approximate nonlinearity and the actual nonlinearity is estimated using the sub-predictor estimation error. We  prove that for any $r>0$, the  LMIs for the stability analysis are feasible provided $M$ and $N$ are large enough and the Lipschitz constant is small enough. Numerical examples demonstrate the efficiency of the proposed approach. 

\textit{Notations and preliminaries:} $ L^2(0,1)$ is the Hilbert space of Lebesgue measurable and square integrable functions $f:[0,1]\to \mathbb{R} $  with the inner product $\left< f,g\right>:=\scriptsize{\int_0^1 f(x)g(x)dx}$ and induced norm $\left\|f \right\|^2:=\left<f,f \right>$.
$H^{k}(0,1)$ is the Sobolev space of  functions  $f:[0,1]\to \mathbb{R} $ having $k$ square integrable weak derivatives, with the norm $\left\|f \right\|^2_{H^k}:=\sum_{j=0}^{k} \left\|f^{(j)} \right\|^2$. The Euclidean norm on $\mathbb{R}^n$ is denoted by $\left|\cdot \right|$. We write $f\in H^1_0(0,1)$ if $f\in H^1(0,1)$ and $f(0)=f(1)=0$. For  $P \in \mathbb{R}^{n \times n}$, $P>0$ means that $P$ is symmetric and positive definite.
The sub-diagonal elements of a symmetric matrix will be denoted by $*.$ For $0<U\in \mathbb{R}^{n\times n}$ and $x\in \mathbb{R}^n$ we denote $\left|x\right|^2_U=x^TUx$. $\mathbb{Z}_+$ denotes the nonnegative integers.

Consider the Sturm-Liouville eigenvalue problem
\begin{equation}\label{eq:SL}
	\begin{aligned}
		\phi''+\lambda \phi = 0,\ \ x\in (0,1)
	\end{aligned}
\end{equation}
with boundary conditions
\begin{equation}\label{eq:2BCs}
	\begin{array}{lll}
		& \phi'(0) = \phi'(1) = 0.
	\end{array}
\end{equation}
This problem induces a sequence of eigenvalues with corresponding eigenfunctions. The normalized eigenfunctions form a complete orthonormal system in $L^2(0,1)$. The eigenvalues and corresponding eigenfunctions are given by
\begin{equation}\label{eq:SLBCs}
	\begin{array}{lll}
		&\hspace{-5mm}\phi_0(x) \equiv 1,\ \phi_n(x)=\sqrt{2}\cos\left(\sqrt{\lambda_n} x\right),\ \lambda_n = n^2\pi^2, \ n\in \mathbb{Z}_+.
	\end{array}
\end{equation}
The following lemmas will be used:
\begin{lemma}\label{lem:H1Equiv1}  
Let $h \overset{L^2}{=} \sum_{n=0}^{\infty}h_n\phi_n$. Then $h\in H^2(0,1)$ with $h'(0)=h'(1)=0$ if and only if
$\sum_{n=1}^{\infty}\lambda_n^2h_n^2< \infty$.  Moreover,
\begin{equation}\label{eq:H2}
\left\|h' \right\|^2=\sum_{n=1}^{\infty}\lambda_nh_n^2.
\end{equation}
\end{lemma}
\begin{lemma}\label{lem_sob}
(Sobolev's inequality \cite{kang2019distributed})
	Let	$h\in H^1(0,1)$. Then, for all $\Gamma>0$ :
	\begin{equation}\label{eq:Sob}
	    max_{x\in[0,1]}\left|h(x) \right|^2\leq (1+\Gamma)\left\|h\right\|^2+ \Gamma^{-1}\left\|h' \right\|^2.
	\end{equation}
\end{lemma}

\section{Finite-dimensional observer-based control of a non-delayed semilinear heat equation}\label{sec_ObsContSemilinNodel}

\subsection{Problem formulation and controller deign}
In this section we consider stabilization of the non-delayed semilinear 1D heat equation
\begin{equation}\label{eq:SemilinearHeatNoDel}
\begin{array}{lll}
&z_t(x,t) = z_{xx}(x,t) +g\left(t,x,z(x,t)\right), \ t\geq 0
\end{array}
\end{equation}
where $x\in [0,1]$, $z(x,t)\in \mathbb{R}$ and $a\in \mathbb{R}$ is the reaction coefficient. We consider Neumann actuation
\begin{equation}\label{eq:SemiBCsHeatDir}
z_x(0,t) = 0, \quad z_x(1,t) = u(t)
\end{equation}
where $u(t)$ is a control input to be designed. We further assume point measurement given by
\begin{equation}\label{eq:SemiInDomPointMeasHeat}
y(t) = z(x_*,t), \quad x_*\in [0,1].
\end{equation}
Note that $x_*=0$ or $x_*=1$ correspond to boundary measurements. Here $g:\mathbb{R}^3 \to \mathbb{R}$ is a locally Lipschitz function which satisfies $g(t,x,0) \equiv0$  and
\begin{equation}\label{eq:NonLinHeatVectfNoDel}
\begin{array}{lll}
&\sup_{z_1\neq z_2}\frac{\left|g(t,x,z_1)-g(t,x,z_2) \right|}{\left|z_1-z_2 \right|}\leq \sigma, \ \forall \left(t,x\right)\in \mathbb{R}^2
\end{array}
\end{equation}
for some $\sigma>0$, independent of $\left(t,x\right)\in \mathbb{R}^2$.

\begin{remark}
	For simplicity, in the present paper  we consider a reaction-diffusion PDE with constant diffusion and reaction coefficients. As in \cite{RamiContructiveFiniteDim}, our results can be easily extended to the more general reaction-diffusion PDE
	\begin{equation*}
	\begin{array}{lll}
	&z_t(x,t) = \partial_x\left(p(x)z_x(x,t)\right)+q(x)z(x,t)\\
	&\hspace{12mm}+g(t,x,z(x,t)),\quad x\in[0,1], \ t\ge 0,
	\end{array}
	\end{equation*}
	where $p(x)$ and $q(x)$ are sufficiently smooth on $(0,1)$.
\end{remark}

Let $\psi(x) = -\frac{2}{\pi}\cos\left(\frac{\pi}{2}x\right)$ and note that it satisfies
\begin{equation}\label{eq:PsiProp}
\begin{array}{lll}
    & \psi''(x) = -\mu \psi(x), \quad \mu = \frac{\pi^2}{4},\\
    & \psi'(0) = 0 , \ \psi'(1) = 1, \ \left\|\psi \right\|^2 = \frac{2}{\pi^2}.
\end{array}
\end{equation}
Forthermore, note that
\begin{equation}\label{eq:PsiProp1}
\begin{array}{lll}
& \left<\psi,\phi_0 \right> = \int_0^1\psi(x)dx = \frac{4}{\pi^2},\\
& \left<\psi,\phi_n \right> = -\frac{1}{\lambda_n}\int_0^1\psi(x)\phi_n''(x)dx = \frac{1}{\lambda_n}\phi'_n(1)\\
&-\frac{1}{\lambda_n}\int_0^1\psi''(x)\phi_n(x)dx= \frac{\sqrt{2}(-1)^{n}}{\lambda_n}+\frac{\mu}{\lambda_n}\left<\psi,\phi_n \right>, \ n\geq 1.
\end{array}
\end{equation}
Similar to  \cite{karafyllis2021lyapunov}, we introduce the change of variables
\begin{equation}\label{eq:SemiChangeVarsHeatDirNoDel}
w(x,t)=z(x,t)-\psi(x)u(t),
\end{equation}
to obtain the equivalent PDE
\begin{equation}\label{eq:wnodel}
\begin{array}{lll}
& w_t(x,t)=w_{xx}(x,t)+g\left(t,x,w(x,t)+\psi(x)u(t) \right)\\
&\hspace{12mm} -\psi(x)[\dot{u}(t) +\mu u(t)]
\end{array}
\end{equation}
with
\begin{equation}\label{eq:SemiBCsHeatDir1}
w_x(0,t) = w_x(1,t) = 0
\end{equation}
and measurement
\begin{equation}\label{eq:SemiInDomPointMeas1HeatDirNoDel}
y(t) = w(x_*,t)+\psi(x_*)u(t).
\end{equation}
 We define further the new control input $v(t)$ that satisfies the following relations:
$$\dot{u}(t) = -\mu u(t) + v(t),\quad u(0)=0, \quad t\ge 0.$$
 Then \eqref{eq:wnodel} can be presented as the 
 ODE-PDE system 
\begin{equation}\label{eq:SemiPDE1HeatNeuSampNoDel}
\begin{array}{lll}
& \dot{u}(t) = -\mu u(t) + v(t),\quad t\ge 0, \\
& w_t(x,t)=w_{xx}(x,t)+g\left(t,x,w(x,t)+\psi(x)u(t) \right)\\
&\hspace{12mm} -\psi(x)v(t).
\end{array}
\end{equation}
We will  treat further $u(t)$ as an additional state variable.  

We present the solution to \eqref{eq:SemiPDE1HeatNeuSampNoDel} as
\begin{equation}\label{eq:SemiWseriesHeatDirNoDel}
\begin{array}{lll}
w(x,t) &= \sum_{n=0}^{\infty}w_n(t)\phi_n(x), \ w_n(t) =\left<w(\cdot,t),\phi_n\right>,
\end{array}
\end{equation}
with $\left\{\phi_n\right\}_{n=0}^{\infty}$ defined in \eqref{eq:SLBCs}. By differentiating under the integral sign, integrating by parts and using \eqref{eq:SL} and \eqref{eq:2BCs} we obtain for $t\geq 0$
\begin{equation}\label{eq:SemiWOdesHeatDirNoDel}
\begin{array}{lll}
&\dot{w}_n(t) = -\lambda_nw_n(t) + g_n(t)+ b_nv(t), \\
&w_n(0) = \left<w(\cdot,0),\phi_n\right>,
\end{array}
\end{equation}
where
\begin{equation}\label{eq:gnDef}
\begin{array}{lll}
&g_n(t) = \left<g\left(t,\cdot,w(\cdot,t)+\psi(\cdot)u(t) \right),\phi_n \right>,\\
& b_0 \overset{\eqref{eq:PsiProp1}}{=}\frac{4}{\pi^2},\quad b_n = \overset{\eqref{eq:PsiProp1}}{=}\frac{(-1)^{n+1}4\sqrt{2}}{\pi^2\left(4n^2-1\right)}, \ n\geq 1.
\end{array}
\end{equation}
Note that given $N\in \mathbb{Z}_+$, \eqref{eq:gnDef} and the integral test for series convergence imply
\begin{equation}\label{eq:bSeriesNoDel}
\begin{array}{lll}
&\sum_{n=N+1}^{\infty}\lambda_nb_n^2 = \frac{32}{\pi^2}\sum_{n=N+1}^{\infty}\frac{n^2}{(4n^2-1)^2}\\
& = \frac{2}{\pi^2}\sum_{n=N+1}^{\infty}\frac{1}{n^2}\left(1+\frac{1}{4n^2-1} \right)^2\leq \frac{2\xi_{N+1}}{\pi^2},\\
&\xi_{N+1} = \left(1+\frac{1}{4(N+1)^2-1} \right)^2\frac{1}{N}.
\end{array}
\end{equation}

Let $\delta>0$ be a desired decay rate and let $N_0 \in \mathbb{Z}_+$ satisfy
\begin{equation}\label{eq:SemiN0HeatDirNoDel}
-\lambda_n+\sigma<-\delta, \quad n>N_0.
\end{equation}
$N_0$ is the number of modes in our controller, whereas  $N\in\mathbb{Z}_+, \ N\geq N_0$ is the  observer dimension.
We construct a finite-dimensional observer of the form
\begin{equation}\label{eq:ObsNoDel}
\begin{array}{lll}
\hat{w}(x,t) = \sum_{n=0}^N\hat{w}_n(t)\phi_n(x)
\end{array}
\end{equation}
where $\left\{\hat{w}_n(t) \right\}_{n=0}^N$ satisfy the \emph{nonlinear} ODEs
\begin{equation}\label{eq:ObsNoDel1}
    \begin{array}{lll}
    &\dot{\hat{w}}_n(t) = -\lambda_n\hat{w}_n(t) + \hat{g}_n(t)+ b_nv(t)\\
    & -l_n\left[\hat{w}(x_*,t)+\psi(x_*)u(t) - y(t) \right], \ 0\leq n \leq N
    \end{array}
\end{equation}
with scalar observer gains $\left\{l_n \right\}_{n=0}^N$ and
\begin{equation}\label{eq:ObsNoDel2}
\begin{array}{lll}
\hat{g}_n(t) = \left <g\left(t,\cdot,\hat{w}(\cdot,t)+\psi(\cdot)u(t) \right),\phi_n \right>, \ 0\leq n \leq N.
\end{array}
\end{equation}
In particular, we approximate the projections of the semilinearity $g(t,x,w(x,t)+\psi(x)u(t))$ onto $\left\{\phi_n \right\}_{n=0}^N$ by the projections of the approximate semilinearity $g(t,x,\hat{w}(x,t)+\psi(x)u(t))$ onto $\left\{\phi_n \right\}_{n=0}^N$.
Assume

\textbf{Assumption 1:} The point $x_*\in [0,1]$ satisfies
\begin{equation}\label{eq:SemiXstarassumpNoDel}
c_n= \phi_n(x_*) \neq 0,\quad 0\leq n\leq N_0.
\end{equation}
Note that Assumption 1 holds for the particular case of boundary measurements $x_*=0$ or $x^*=1$.

Denote
\begin{equation}\label{eq:SemiC0A0HeatDirNoDel}
\begin{array}{lll}
&\tilde{A}_0 = \operatorname{diag}\left\{-\mu,A_0 \right\}, \quad \tilde{B}_0 = \operatorname{col}\left\{1,B_0\right\}\\
&A_0 = \operatorname{diag}\left\{-\lambda_n\right\}_{n=0}^{N_0}, \quad B_0 = \operatorname{col}\left\{b_n\right\}_{n=0}^{N_0}\\
&C_0=\left[c_0,\dots,c_{N_0} \right],\quad  C_1 = \left[c_{N_0+1},\dots,c_{N} \right],\\
\end{array}
\end{equation}
 Under Assumption 1, the pair $(A_0,C_0)$ is observable by the Hautus lemma. Let $L_0 = \left\{l_n \right\}_{n=0}^{N_0} \in \mathbb{R}^{N_0+1}$ satisfy the Lyapunov inequality
\begin{equation}\label{eq:SemiGainsDesignLHeatDirNoDel}
P_{\text{o}}(A_0-L_0C_0)+(A_0-L_0C_0)^TP_{\text{o}} < -2\delta P_{\text{o}}
\end{equation}
with $0<P_{\text{o}}\in \mathbb{R}^{(N_0+1)\times (N_0+1)}$. We further choose the remaining gains as $l_n = 0, \ N_0+1\leq n\leq N$.

Similarly, by the Hautus lemma, the pair $(\tilde{A}_0, \tilde{B}_0)$ is controllable. Let $K_0\in \mathbb{R}^{1\times (N_0+2)}$ satisfy
\begin{equation}\label{eq:SemiGainsDesignKHeatDirNoDel}
\begin{aligned}
&P_{\text{c}}(\tilde{A}_0-\tilde{B}_0K_0)+(\tilde{A}_0-\tilde{B}_0K_0)^TP_{\text{c}} < -2\delta P_{\text{c}},
\end{aligned}
\end{equation}
with $0<P_{\text{c}}\in \mathbb{R}^{(N_0+2)\times (N_0+2)}$.
We propose the 
controller 
\begin{equation}\label{eq:ContDefNoDel}
\begin{array}{lll}
& v(t) = -K_0\hat{w}^{N_0}(t),
& \hat{w}^{N_0}(t) = \operatorname{col}\left\{u(t),\hat{w}_n(t) \right\}_{n=0}^{N_0}
\end{array}
\end{equation}
which is based on the finite-dimensional observer \eqref{eq:ObsNoDel}.

\subsection{Well-posedness of the closed-loop system}
For well-posedness of the closed-loop system \eqref{eq:SemiChangeVarsHeatDirNoDel}, \eqref{eq:ObsNoDel1} subject to the control law \eqref{eq:ContDefNoDel}, consider the operator
\begin{equation*}
	\begin{array}{lll}
		&\hspace{-5mm}\mathcal{A}:\mathcal{D}(\mathcal{A})\to L^2(0,1),\ \mathcal{A} = -\partial_{xx},\\
		&\hspace{-5mm} \mathcal{D}(\mathcal{A}) = \left\{h\in H^2(0,1) \ | \ h'(0)=h'(1)=0\right\}.
	\end{array}
\end{equation*}
Let $\theta>0$ and $\mathcal{A}_{\theta}=\mathcal{A}+\theta I$. Given $h\in \mathcal{D}(\mathcal{A}_{\theta})=\mathcal{D}(\mathcal{A})$, integration by parts gives $\left<\mathcal{A}_{\theta}h,h\right> = \left\|h' \right\|^2+\theta \left\|h\right\|^2$. Hence, $\left<\mathcal{A}_{\theta}h,h\right>>0$. Since $-\mathcal{A}_{\theta}$ is diagonalizable, by Section 2.6 in \cite{GeorgeBook}, the spectrum of $-\mathcal{A}_{\theta}$ is given by $\sigma\left(-\mathcal{A}_{\theta} \right) = \left\{-\lambda_n-\theta \right\}_{n=0}^{\infty}\subset (-\infty,0)$. Thus, $\left\{\mu \in \mathbb{C} \ | \text{Re}(\mu )>0\right\}\subseteq \rho \left(-\mathcal{A}_{\theta} \right)$, where $\rho \left(-\mathcal{A}_{\theta} \right)$ is the resolvent set of $-\mathcal{A}_{\theta}$. By \cite{GeorgeBook},  $-\mathcal{A}_{\theta}$  generates
an analytic semigroup on $L^2(0,1)$. Moreover, by Section 3.4 in \cite{GeorgeBook} and positivity of $\mathcal{A}_{\theta}$, there exists a unique positive root $\mathcal{A}^{\frac{1}{2}}_{\theta}$ where $\mathcal{D}\left(\mathcal{A}^{\frac{1}{2}}_{\theta} \right)\subseteq L^2(0,1)$ is the completion of $\mathcal{D}\left(\mathcal{A}_{\theta} \right)\subseteq L^2(0,1)$ with respect to the norm $\left\|h\right\|_{\frac{1}{2}}=\sqrt{\left<\mathcal{A}_{\theta}h,h\right>}=\sqrt{\left\|h' \right\|^2+\theta \left\|h\right\|^2}$. Hence, $\mathcal{D}\left(\mathcal{A}^{\frac{1}{2}}_{\theta} \right) =H^1(0,1)$. Let $\mathcal{H} = L^2(0,1)\times \mathbb{R}^{N+2}$ be a Hilbert space with the norm $\left\|\cdot\right\|_{\mathcal{H}}^2:=\left\| \cdot\right\|^2+\left|\cdot \right|^2$. Introducing the state
\begin{equation}\label{eq:AbsStatNoDel}
	\begin{array}{lll}
		&\xi(t) = \text{col}\left\{\xi_1(t),\xi_2(t)\right\} , \ \xi_1(t)=w(\cdot,t), \ \xi_2(t) = \hat{w}^N(t),\\
		&\hat{w}^N(t) = \operatorname{col}\left\{u(t),\hat{w}_0(t),\dots,\hat{w}_N(t) \right\}
	\end{array}
\end{equation}
the closed-loop system can be presented as
\begin{equation}\label{eq:AbstractODENoDel}
	\begin{array}{lll}
		&\frac{d\xi }{dt}(t)+\operatorname{diag}\left\{\mathcal{A}_{\theta},\mathcal{B}\right\}\xi(t)=\begin{bmatrix}f_1(\xi)\\f_2(\xi)\end{bmatrix},\\
		&\mathcal{D}\left(\mathcal{B} \right)=\mathbb{R}^{N+2}, \ \mathcal{B}a = \scriptsize\begin{bmatrix}
		-\tilde{A}_0+\tilde{B}_0K_0+\tilde{L}_0[0\ C_0] & \tilde{L}_0C_1\\
		B_1K_0 & -A_1
		\end{bmatrix}a\\
		& f_1(t,\xi) = \theta w(\cdot,t)+g\left(t,\cdot,w(\cdot,t)+\psi(\cdot)u(t) \right)\\
		&\hspace{12mm}+\psi(\cdot)K_0\hat{w}^{N_0}(t),\\
		&f_2(t,\xi) = \operatorname{col}\left\{\hat{G}^{N_0}(t)+\tilde{L}_0w(x_*,t), \hat{G}^{N-N_0}(t)\right\}
	\end{array}
\end{equation}
where $-\mathcal{B}$ generates an analytic semigroup on $\mathcal{H}$ and
\begin{equation}\label{eq:AbstractODENoDel1}
	\begin{array}{lll}
&\hat{G}^{N_0}(t) = \operatorname{col}\left\{0,\hat{g}_n(t)\right\}_{n=0}^{N_0}, \\
&\hat{G}^{N-N_0}(t) = \operatorname{col}\left\{\hat{g}_n(t)\right\}_{n=N_0+1}^N , \ \tilde{L}_0 = \operatorname{col}\left\{0,l_n \right\}_{n=0}^{N_0}, \\
&A_1 = \operatorname{diag}\left\{-\lambda_{n}\right\}_{n=N_0+1}^N,\ B_1 = \operatorname{col}\left\{b_n\right\}_{n=N_0+1}^N.
\end{array}
\end{equation}
Let $\mathcal{G} = H^1(0,1)\times \mathbb{R}^{N+2}$ be a Hilbert space with the norm $\left\|\cdot\right\|_{\mathcal{G}}^2:=\left\| \cdot\right\|_{H^1}^2+\left|\cdot \right|^2$. Fix $(t,\xi)\in [0,\infty)\times \mathcal{G}$. Let $\mathcal{Q} = J\times B_{\mathcal{G}}(\xi,R)$ be a neighborhood of $(t,\xi)$, where $J$ is an interval and $B_{\mathcal{G}}(\xi,R)$ is a ball of radius $R>0$ around $\xi$. Let $(t_j,\varphi^{(j)})\in \mathcal{Q}, \ j\in \left\{1,2 \right\}$. Fixing $\Gamma=1$, by the Sobolev inequality, for any $j\in \left\{1,2 \right\}$ we have
\begin{equation}\label{eq:AbstractODENoDel11}
\begin{array}{lll}
&\max_{x\in [0,1]}\left|\varphi^{(j)}_1(x)\right|^2\overset{\eqref{eq:Sob}}{\leq} 2\left\|\varphi^{(j)}_1 \right\|_{H^1}^2\leq 2\left(R + \left\|\xi_1 \right\|_{H^1} \right)^2,\\
&\max_{x\in [0,1]}\left|[\psi(x)\ 0]\varphi^{(j)}_2\right|^2\leq \left\|\psi(x) \right\|_{\infty}^2\left(R+ \left|\xi_2 \right| \right)^2.
\end{array}
\end{equation}
Hence, for some $R_1(\xi)>0$ we have for $j\in \left\{1,2 \right\}$ that $\max_{x\in [0,1]}\left|\varphi^{(j)}_1(x)-[\psi(x)\ 0]\varphi^{(j)}_2\right|\leq R_1(\xi)$. Let $\mathcal{S} = \operatorname{cl}\left(J\right)\times [0,1]\times [-R_1(\xi),R_1(\xi)]\subseteq \mathbb{R}^3$. By assumption, $g$ is locally Lipschitz. Denote by $L_\mathcal{S}$ it's Lipschitz constant on $\mathcal{S}$. Then, we obtain
\begin{equation}\label{eq:AbstractODENoDel1111}
\begin{array}{lll}
&\left\|g(t_1,\cdot,\varphi_1^{(1)}(\cdot)+[\psi(\cdot) \ 0]\varphi^{(1)}_2)\right.\\
&\hspace{10mm}\left. -g(t_2,\cdot,\varphi_1^{(2)}(\cdot)+[\psi(\cdot) \ 0]\varphi^{(2)}_2) \right\|^2\\
&\leq 2L_S^2 \left(|t_1-t_2|^2+\left\|\varphi^{(1)}-\varphi^{(2)} \right\|^2_{\mathcal{G}}\right)
\end{array}
\end{equation}
From \eqref{eq:Sob}, \eqref{eq:AbstractODENoDel} and \eqref{eq:AbstractODENoDel1111} it easily follows that $f_1(t,\xi)$ and $f_2(t,\xi)$ satisfy assumption (F) in Thereom 6.3.1 in \cite{pazy1983semigroups}. Furthermore, by \eqref{eq:NonLinHeatVectfNoDel}, $f_1(t,\xi)$ and $f_2(t,\xi)$ also satisfy the conditions of Thereom 6.3.3 in \cite{pazy1983semigroups}. Hence, given $w(\cdot,0)\in H^1(0,1)$, the system \eqref{eq:AbstractODENoDel} has a unique classical solution satisfying
\begin{equation}\label{eq:xisolNoDel}
\begin{array}{lll}
&\xi \in C\left([0,\infty);\mathcal{H}\right)\cap C^1\left( (0,\infty);\mathcal{H}\right)
\end{array}
\end{equation}
such that
\begin{equation}\label{eq:xisolNoDel1}
\xi(t)\in \mathcal{D}\left(\operatorname{diag}\left\{\mathcal{A}_{\theta},\mathcal{B}\right\}\right)= \mathcal{D}\left(\mathcal{A}\right)\times \mathbb{R}^{N+2}\quad \forall t>0.
\end{equation}
\subsection{Stability analysis of the closed-loop system}
Introduce the estimation error $e_n(t) = w_n(t)-\hat{w}_n(t), \ 0\leq n\leq N_0$. Using the estimation error and $\left\{c_n\right\}_{n=0}^N$ in \eqref{eq:SemiC0A0HeatDirNoDel}, the innovation term in \eqref{eq:ObsNoDel1} can be presented as
\begin{equation}\label{eq:InnovNoDel}
\begin{array}{lll}
    &\hat{w}(x_*,t)+\psi(x_*)u(t) - y(t) = \hat{w}(x_*,t) - w(x_*,t)\\
    & = -\sum_{n=0}^Nc_ne_n(t) - \zeta(t),\\
    &\zeta(t) = w(0,t)-\sum_{n=0}^Nw_n(t)\phi_n.
\end{array}
\end{equation}
Let $\Gamma>0$. By Lemma \ref{lem_sob} we have
\begin{equation}\label{eq:ZetaEst}
\begin{array}{lll}
    &\zeta^2(t) \leq (1+\Gamma)\left\|w(\cdot,t)-\sum_{n=0}^Nw_n(t)\phi_n(\cdot)\right\|^2\\
    &\hspace{12mm}+ \Gamma^{-1}\left\|w_x(\cdot,t)-\sum_{n=0}^Nw_n(t)\phi'_n(\cdot) \right\|^2\\
    &\hspace{12mm}\overset{\eqref{eq:H2}}{=}\sum_{n=N+1}^{\infty}\kappa_nw_n^2(t), \ \kappa_n = 1+\Gamma +\Gamma^{-1}\lambda_n.
\end{array}
\end{equation}
Taking into account \eqref{eq:SemiWOdesHeatDirNoDel}, \eqref{eq:ObsNoDel1}, \eqref{eq:SemiC0A0HeatDirNoDel} and \eqref{eq:InnovNoDel}, the estimation error satisfies the following ODEs
\begin{equation}\label{eq:EstErrNoDel}
\begin{array}{lll}
&\dot{e}_n(t) = -\lambda_ne_n(t) + h_n(t)\\
&\hspace{10mm}- l_n\sum_{n=0}^Nc_ne_n(t)-l_n\zeta(t), \ 0\leq n \leq N_0,\\
&\dot{e}_n(t) =-\lambda_ne_n(t) + h_n(t), \ N_0+1\leq n\leq N.
\end{array}
\end{equation}
where we define
\begin{equation}\label{eq:EstErrNoDel1}
    h_n(t) = g_n(t) -\hat{g}_n(t), \quad n\geq 0.
\end{equation}

Recall \eqref{eq:SemiC0A0HeatDirNoDel}, \eqref{eq:AbstractODENoDel1} and denote
\begin{equation}\label{eq:SemiC0A0HeatDir1NoDel}
\begin{array}{lll}
&\hat{w}^{N-N_0}(t) = \operatorname{col}\left\{\hat{w}_{n}(t)\right\}_{n=N_0+1}^N,\\
&e^{N_0}(t) = \operatorname{col}\left\{e_{n}(t)\right\}_{n=0}^{N_0},\\ &e^{N-N_0}(t) = \operatorname{col}\left\{e_{n}(t)\right\}_{n=N_0+1}^{N},\\
&H^{N_0}(t) = \operatorname{col}\left\{h_n(t)\right\}_{n=0}^{N_0},  \\
&H^{N-N_0}(t) = \operatorname{col}\left\{h_n(t)\right\}_{n=N_0+1}^N ,\\
&X(t) = \operatorname{col}\left\{\hat{w}^{N_0}(t),e^{N_0}(t),\hat{w}^{N-N_0}(t),e^{N-N_0}(t) \right\},\\
&L_{\zeta} = \operatorname{col}\left\{\tilde{L}_0,-L_0,0,0 \right\}\in \mathbb{R}^{2N+3},		\\
&\hat{G}(t) = \operatorname{col}\left\{\hat{G}^{N_0}(t),0,\hat{G}^{N-N_0}(t),0 \right\},\\
&H(t) = \operatorname{col}\left\{0, H^{N_0}(t),0,H^{N-N_0}(t)  \right\},\\
& K_X = [K_0 , 0 , 0 ,0]\in \mathbb{R}^{1\times (2N+3))}.
\end{array}
\end{equation}
Then, using \eqref{eq:SemiWOdesHeatDirNoDel}, \eqref{eq:ObsNoDel1} - \eqref{eq:SemiC0A0HeatDirNoDel}, \eqref{eq:ContDefNoDel}, \eqref{eq:InnovNoDel}, \eqref{eq:EstErrNoDel} and \eqref{eq:SemiC0A0HeatDir1NoDel}, the closed-loop system for $t\geq 0$ can be presented as
\begin{equation}\label{eq:ClosedLoopNoDel}
\begin{array}{lll}
     &  \dot{X}(t) = F_XX(t)+L_{\zeta}\zeta(t)+\hat{G}(t)+H(t),\\
     &\dot{w}_n(t) = -\lambda_nw_n(t)+\hat{g}_n(t)+h_n(t)\\
     &\hspace{13mm}-b_nK_XX(t), \ n>N
\end{array}
\end{equation}
where
\begin{equation*}
F_X = \scriptsize \begin{bmatrix}
	\tilde{A}_0-\tilde{B}_0K_0 & \tilde{L}_0C_0 & 0& \tilde{L}_0C_1\\
	0 & A_0-L_0C_0 & 0 & -L_0C_1\\
	-B_1K_0 & 0 & A_1 & 0\\
	0 & 0 & 0 & A_1
\end{bmatrix}.
\end{equation*}
For $H^1$-stability analysis of the closed-loop system \eqref{eq:ClosedLoopNoDel} we consider the Lyapunov function
\begin{equation}\label{eq:Lyapunov}
V(t) = X^T(t)P_XX(t)+\sum_{n=N+1}^{\infty}\lambda_nw_n^2(t)
\end{equation}
where $0<P_X\in \mathbb{R}^{(2N+3)\times (2N+3)}$ to be obtained from LMIs. Differentiating $V(t)$ along the solution to the closed-loop system \eqref{eq:ClosedLoopNoDel} we have
\begin{equation}\label{eq: VdotNoDel}
\begin{array}{lll}
&\dot{V}+2\delta V = 2X^T(t)\left[P_XF_X+F_X^TP_X+2\delta P_X \right]X(t)\\
&+ 2X^T(t)P_XL_{\zeta}\zeta(t)+2X^T(t)P_X\hat{G}(t)+2X^T(t)P_XH(t)\\
&+2\sum_{n=N+1}^{\infty}\left(-\lambda_n^2+\delta \lambda_n \right)w_n^2(t)\\
&+ 2\sum_{n=N+1}^{\infty}\lambda_nw_n(t)\left[\hat{g}_n(t)+h_n(t)-b_NK_XX(t) \right].
\end{array}
\end{equation}
Let $\alpha_1>0$, we compensate the series with $\left\{\hat{g}_n(t) \right\}_{n=N+1}^{\infty}$ by using the Young inequality
\begin{equation}\label{eq: VdotNoDel1}
\begin{array}{lll}
&2\sum_{n=N+1}^{\infty}\lambda_nw_n(t)\hat{g}_n(t)\leq \frac{1}{\alpha_1}\sum_{n=N+1}\lambda_n^2w_n^2(t)\\
&-\alpha_1\left|\hat{G}(t) \right|^2+\alpha_1 \sum_{n=0}^{\infty}\hat{g}_n^2(t).
\end{array}
\end{equation}
Then, by Parseval's equality and \eqref{eq:NonLinHeatVectfNoDel} we obtain
\begin{equation}\label{eq: VdotNoDel2}
\begin{array}{lll}
&\alpha_1 \sum_{n=0}^{\infty}\hat{g}_n^2(t) = \alpha_1 \int_0^1|g(t,x,\hat{w}(x,t)+\psi(x)u(t))|^2 dx\\
&\overset{\eqref{eq:NonLinHeatVectfNoDel}}{\leq} \alpha_1\sigma^2\int_0^1|\hat{w}(x,t)+\psi(x)u(t)) |^2dx\\
&\leq 2\alpha_1\sigma^2\left\|\hat{w}(\cdot,t) \right\|^2 + 2\alpha_1 \sigma^2u^2(t)\left\|\psi \right\|^2\\
&= 2\alpha_1\sigma^2X^T(t)\Xi_XX(t),\\
&\Xi_X \overset{\eqref{eq:PsiProp}}{=} \operatorname{diag}\left\{\frac{2}{\pi^2}, I_{N_0+1},0,I_{N-N_0},0 \right\}.
\end{array}
\end{equation}
Similarly, introducing $\alpha_2>0$ we have
\begin{equation}\label{eq: VdotNoDel3}
\begin{array}{lll}
&2\sum_{n=N+1}^{\infty}\lambda_nw_n(t)h_n(t)\leq \frac{1}{\alpha_2}\sum_{n=N+1}\lambda_n^2w_n^2(t)\\
&-\alpha_2\left|H(t) \right|^2+\alpha_2 \sum_{n=0}^{\infty}h_n^2(t).
\end{array}
\end{equation}
Recall that
\begin{equation}\label{eq:eq: VdotNoDel31}
\begin{array}{lll}
&h_n = \left<g\left(t,\cdot,w(\cdot,t)+\psi(\cdot)u(t) \right),\phi_n\right>\\
&\hspace{15mm}-\left< g\left(t,\cdot,\hat{w}(\cdot,t)+\psi(\cdot)u(t) \right),\phi_n\right>, \ n\geq 0.
\end{array}
\end{equation}
Then, by Parseval's equality we obtain
\begin{equation}\label{eq: VdotNoDel4}
\begin{array}{lll}
&\alpha_2 \sum_{n=0}^{\infty}h_n^2(t)\overset{\eqref{eq:NonLinHeatVectfNoDel}}{\leq} \alpha_2\sigma^2\int_0^1|\hat{w}(x,t)-w(x,t) |^2dx\\
& = \alpha_2\sigma^2X^T(t)\Xi_EX(t)+\alpha_2\sigma^2\sum_{n=N+1}w_n^2(t),\\
&\Xi_E \overset{}{=} \operatorname{diag}\left\{0, I_{N_0},0,I_{N-N_0} \right\}\in \mathbb{R}^{(2N+3)\times(2N+3)}.
\end{array}
\end{equation}
We bound the last term in \eqref{eq: VdotNoDel}
by using Young's inequality with some $\alpha_3>0$:
\begin{equation}\label{eq: VdotNoDel5}
\begin{array}{lll}
&2\sum_{n=N+1}^{\infty} \lambda_nw_n(t)\left(-b_nK_XX(t) \right)\\
&\leq \frac{1}{\alpha_3}\sum_{n=N+1}^{\infty}\lambda_nw_n^2(t)+\alpha_3\left(\sum_{n=N+1}^{\infty}\lambda_nb_n^2 \right)\left|K_XX(t)\right|^2\\
&\overset{\eqref{eq:bSeriesNoDel}}{\leq}\frac{1}{\alpha_3}\sum_{n=N+1}^{\infty}\lambda_nw_n^2(t)+\frac{2\alpha_3\xi_{N+1}}{\pi^2}\left|K_XX(t)\right|^2.
\end{array}
\end{equation}
Finally, 
denoting for $n\ge N$
\begin{equation*}\label{eq:rho_n}\rho_{n} = \kappa_n^{-1}\left(-\lambda_n^2+\delta \lambda_n+\frac{\lambda_n}{2\alpha_3}+\frac{\lambda_n^2}{2\alpha_2}+\frac{\lambda_n^2}{2\alpha_2}+\frac{\alpha_2\sigma^2}{2} \right)\end{equation*}
and assuming that $\rho_{N+1}< 0$,
it can be seen that $\rho_n$ is monotonically decreasing. The latter follows from monotonicity of $\lambda_n$. Then
for the series terms in \eqref{eq: VdotNoDel} we have 
\begin{equation}\label{eq: VdotNoDel6}
\begin{array}{lll}
&\sum_{n=N+1}^{\infty} \left(-\lambda_n^2+\delta \lambda_n+\frac{\lambda_n}{2\alpha_3}+\frac{\lambda_n^2}{2\alpha_1}+\frac{\lambda_n^2}{2\alpha_2}+\frac{\alpha_2\sigma^2}{2} \right)w_n^2(t)\\
& \overset{}{=} \sum_{n=N+1}^{\infty}\rho_n\kappa_n w_n^2(t) 
\overset{\eqref{eq:ZetaEst}}{\leq} \rho_{N+1}\zeta^2(t).
\end{array}
\end{equation}

Let $\eta(t) = \operatorname{col}\left\{X(t),\zeta(t), \hat{G}(t), H(t) \right\}$. From \eqref{eq: VdotNoDel}-\eqref{eq: VdotNoDel6} 
we have
\begin{equation}\label{eq:VdotUpperBd}
\begin{array}{lll}
&\dot{V}+2\delta V \leq \eta^T(t)\Psi_0\eta(t)\leq 0
\end{array}
\end{equation}
provided
\begin{equation}\label{eq:PsiNoDel}
\begin{array}{lll}
&\Psi_0 = \scriptsize\left[
\begin{array}{c|c}
\begin{matrix}\psi_0 & P_XL_{\zeta}\\
* & 2\rho_{N+1}\end{matrix}& \begin{matrix}
P_X \quad & P_X\\
0 \quad & 0
\end{matrix}\\
\hline
* & \operatorname{diag}\left\{-\alpha_1 I, -\alpha_2I\right\}
\end{array}
\right]<0,\\
&\psi_0 = \scriptsize P_XF_X+F_X^TP_X+2\delta P_X+\frac{2\alpha_3\xi_{N+1}}{\pi^2}K_X^TK_X\\
&\hspace{8mm}+2\alpha_1\sigma^2\Xi_X +\alpha_2\sigma^2\Xi_E
\end{array}
\end{equation}

By Schur complement, it can be  seen that $\Psi<0$ is equivalent to the following LMI
\begin{equation}\label{eq:PsiNoDel1}
\begin{array}{lll}
&\scriptsize \left[
\begin{array}{c|c|c}
\begin{matrix}
\psi_0 & P_XL_{\zeta}\\
* & 2\bar{\rho}_{N+1}
\end{matrix}
& \begin{matrix}
P_X \quad & P_X\\
0 \quad & 0
\end{matrix}
& \Pi_1\\
\hline
* & \operatorname{diag}\left\{-\alpha_1 I, -\alpha_2I\right\} & 0\\
\hline
* & * & \Pi_2
\end{array}
\right]<0,\ \Pi_1 = \scriptsize \begin{bmatrix}
	0 \\ 1
\end{bmatrix}\begin{bmatrix}1 & 1 & 1 \end{bmatrix},\\
& \Pi_2 = -\frac{2\kappa_{N+1}}{\lambda_{N+1}}\operatorname{diag}\left\{\frac{\alpha_1}{\lambda_{N+1}},\frac{\alpha_2}{\lambda_{N+1}},\alpha_3 \right\},\\
&\bar{\rho}_{N+1} = 2\kappa_{N+1}^{-1}\left(-\lambda_{N+1}^2+\delta \lambda_{N+1}+\frac{\alpha_2\sigma^2}{2}\right)
\end{array}
\end{equation}
Summarizing, we arrive at
\begin{theorem}\label{Thm:SemilinearNoDel}
Consider the system \eqref{eq:SemiPDE1HeatNeuSampNoDel} with boundary conditions \eqref{eq:SemiBCsHeatDir1}, point measurement \eqref{eq:SemiInDomPointMeas1HeatDirNoDel} and control law \eqref{eq:ContDefNoDel}. Assume that $g(t,x,z)$ is a locally Lipschitz function satisfying $g(t,x,0)\equiv 0$ and \eqref{eq:NonLinHeatVectfNoDel} for a given $\sigma > 0$. Let $\delta>0$, $N_0\in \mathbb{N}$ satisfy \eqref{eq:SemiN0HeatDirNoDel} and $N\in \mathbb{N}$ satisfy $N_0\leq N$. Let $L_0$ and $K_0$ be obtained using \eqref{eq:SemiGainsDesignLHeatDirNoDel} and \eqref{eq:SemiGainsDesignKHeatDirNoDel}, respectively. Given $\Gamma>0$, let there exist $0<P\in \mathbb{R}^{(2N+3)\times (2N+3)}$ and scalars $\alpha_1,\alpha_2,\alpha_3>0$ such that \eqref{eq:PsiNoDel1} holds with $\psi_0$ given in \eqref{eq:PsiNoDel}. Then, given $w(\cdot,0)\in H^1(0,1)$, the solution $u(t),w(x,t)$ of \eqref{eq:SemiPDE1HeatNeuSampNoDel} subject to the control law \eqref{eq:ContDefNoDel} and the observer $\hat{w}(x,t)$ defined by \eqref{eq:ObsNoDel}-\eqref{eq:ObsNoDel2}, satisfy
\begin{equation}\label{eq:StabNoDel}
\begin{array}{lll}
&u^2(t)+\left\|w(\cdot,t) \right\|_{H^1}^2+ \left\|\hat{w}(\cdot,t) \right\|_{H^1}^2\leq D e^{-2\delta t}\left\|w(\cdot,0) \right\|_{H^1}^2
\end{array}
\end{equation}
for $t \geq 0$ and some $D\geq 1$. Moreover, the LMI \eqref{eq:PsiNoDel1} is always feasible for $N$ large enough and $\sigma>0$ small enough.
\end{theorem}
\begin{proof}
Feasibility of \eqref{eq:PsiNoDel1} implies, by the comparison principle, that
$V(t)\leq e^{-2\delta t}V(0), \quad t\geq 0.$
Since $u(0) = 0$ (see \eqref{eq:SemiBCsHeatDir1}) we have
\begin{equation}\label{eq:Comp1NoDel1}
\begin{array}{lll}
&V(0)\leq \sigma_{\text{max}}(P_X)\left[w_0^2(0)+ \sum_{n=1}^{N}w_n^2(0)\right]\\
&+\sum_{n=N+1}^{\infty}\lambda_nw_n^2(0)\overset{\eqref{eq:H2}}{\leq} \operatorname{max}\left\{\sigma_{\text{max}}(P_X),1 \right\}\left\|w(\cdot,0) \right\|^2_{H^1}.
\end{array}
\end{equation}
Similarly for $t\geq0$
\begin{equation}\label{eq:Comp1NoDel2}
\begin{array}{lll}
&V(t)\overset{\eqref{eq:H2}}{\geq}\frac{1}{2}\operatorname{min}\left\{\frac{\sigma_{\text{min}}(P_X)}{\lambda_{N+1}},1 \right\}\left\|w(\cdot,t) \right\|^2_{H^1}.
\end{array}
\end{equation}
The estimate \eqref{eq:StabNoDel} now follows from \eqref{eq:Comp1NoDel1} and \eqref{eq:Comp1NoDel2}. Next, we treat feasibility of \eqref{eq:PsiNoDel1} for large enough $N$ and small enough $\sigma>0$. First, note that for $\sigma =0$ (i.e. when $g\equiv 0$ in \eqref{eq:SemilinearHeatNoDel}) arguments similar to proof of Theorem 3.1 in \cite{RamiContructiveFiniteDim} show feasibility of \eqref{eq:PsiNoDel1} for large enough $N$. Fixing such $N$ and using continuity of the eigenvalues of the matrix in \eqref{eq:PsiNoDel1} we find that \eqref{eq:PsiNoDel1} is feasible for small enough $\sigma>0$.
\end{proof}
\section{Finite-dimensional sequential sub-predictors for  semilinear heat equation}\label{sec_ObsContSemilinDel}
\subsection{Problem formulation}
In this section we consider stabilization of 
\eqref{eq:SemilinearHeatNoDel} under the point measurement \eqref{eq:SemiInDomPointMeasHeat} and subject to delayed Neumann actuation
\begin{equation}\label{eq:SemiBCsHeatDirdel}
z_x(0,t) = 0, \quad z_x(1,t) = u(t-r),\quad t\ge 0.
\end{equation}
Here $r>0$ is a known constant input delay and $u(t)=0$ for $t\le 0$.
 As in the previous section, $g(t,x,z)$ is a locally Lipschitz function satisfying $g(t,x,0)\equiv 0$ and \eqref{eq:NonLinHeatVectfNoDel} for some $\sigma>0$. We aim  to achieve $H^1$-stabilization of \eqref{eq:SemilinearHeatNoDel} in the presence of the input delay $r>0$ in \eqref{eq:SemiBCsHeatDirdel}.

Let $\psi(x) = -\frac{2}{\pi}\cos\left(\frac{\pi}{2}x\right)$ satisfy \eqref{eq:PsiProp} and \eqref{eq:PsiProp1}. To obtain homogeneous boundary conditions we employ the delayed change of variables
\begin{equation}\label{eq:SemiChangeVarsHeatDirDel}
	w(x,t)=z(x,t)-\psi(x)u(t-r),
\end{equation}
that leads to the following PDE
\begin{equation}\label{eq:SemiPDE1HeatNeuSampDel0}
	\begin{array}{lll}
		& w_t(x,t)=w_{xx}(x,t)+g\left(t,x,w(x,t)+\psi(x)u(t-r) \right)\\
		&\hspace{12mm} -\psi(x)\left[\mu u(t-r)+\dot{u}(t-r)\right]
	\end{array}
\end{equation}
As in the non-delayed case, we will construct an integral control law. In order to satisfy $u(t)=0, \ t\leq 0$ and to guarantee that $u(t)$ is continuously differentiable in $t\in \mathbb{R}$,
we consider
\begin{equation}\label{eq:SemiContFrom}
	\begin{array}{lll}
		u(t) = \int_{0}^{t}e^{-\mu (t-s)}v(s)ds, \quad t\in \mathbb{R}
	\end{array}
\end{equation}
where $v(t)$ will be constructed below as continuous subject to $v(t) = 0$ for $t\leq 0$. Then, $u(t)$ satisfies 
\begin{equation}\label{eq:SemiContFrom1}
\begin{array}{lll}
	\dot{u}(t) = -\mu u(t) +v(t), \ t\in \mathbb{R}.
\end{array}
\end{equation}
For our sub-predictor construction below, we would like the ODE for $u$ and the PDE for $w$ to contain the control input evaluated at the same time $t-r$ (see $w^{N_0}(t)$ and $w^{N-N_0}(t)$ in \eqref{eq:SemiStateModsDel} below). Hence, replacing $t$ by $t-r$ in \eqref{eq:SemiContFrom1} and substituting into \eqref{eq:SemiPDE1HeatNeuSampDel0} we obtain the following ODE-PDE system for $t\geq 0$
\begin{equation}\label{eq:SemiPDE1HeatNeuSampDel}
\begin{aligned}
& \dot{u}(t-r) = -\mu u(t-r) + v(t-r), \\
& w_t(x,t)=w_{xx}(x,t)+g\left(t,x,w(x,t)+\psi(x)u(t-r) \right)\\
&\hspace{12mm} -\psi(x)v(t-r)
\end{aligned}
\end{equation}
with the boundary conditions \eqref{eq:SemiBCsHeatDir1} and measurement
\begin{equation}\label{eq:SemiInDomPointMeas1HeatDirDel}
y(t) = w(x_*,t)+\psi(x_*)u(t-r).
\end{equation}
We will treat $u(t-r)$ as the additional (non-delayed) state variable and $v(t-r)$ as the new control input with delay $r$.

We present the solution to \eqref{eq:SemiPDE1HeatNeuSampDel} as \eqref{eq:SemiWseriesHeatDirNoDel}, with $\left\{\phi_n\right\}_{n=0}^{\infty}$ defined in \eqref{eq:SLBCs}. Similar to \eqref{eq:SemiWOdesHeatDirNoDel},
we obtain for $t\geq 0$
\begin{equation}\label{eq:SemiWOdesHeatDirDel}
\begin{array}{lll}
&\dot{w}_n(t) = -\lambda_nw_n(t) + g_n(t) +b_nv(t-r), \\
&w_n(0) = \left<w(\cdot,0),\phi_n\right>, \ n\in \mathbb{Z}_+
\end{array}
\end{equation}
where $\left\{b_n\right\}_{n=0}^{\infty}$ are given in \eqref{eq:gnDef} and
\begin{equation}\label{eq:gnDefDel}
\begin{array}{lll}
&g_n(t) = \left<g\left(t,\cdot,w(\cdot,t)+\psi(\cdot)u(t-r) \right),\phi_n \right>.
\end{array}
\end{equation}
Let $\delta>0$ be a desired decay rate and let $N_0 \in \mathbb{Z}_+$ satisfy \eqref{eq:SemiN0HeatDirNoDel} defining the number of modes in the controller.
Let $N\in \mathbb{Z}_+, \ N\geq N_0$ and introduce
\begin{equation}\label{eq:SemiC0A0HeatDir1Del}
\begin{array}{lll}
&w^{N_0}(t) = \operatorname{col}\left\{u(t-r),w_1(t),\dots, w_{N_0}(t)\right\},\\
&w^{N-N_0}(t) = \operatorname{col}\left\{w_{N_0+1}(t),\dots, w_N(t)\right\},\\
&G^{N_0}(t) = \operatorname{col}\left\{0,g_n(t)\right\}_{n=1}^{N_0} ,\\
&G^{N-N_0}(t) = \operatorname{col}\left\{g_n(t)\right\}_{n=N_0+1}^N.
\end{array}
\end{equation}
Then, recalling $A_1$ and $B_1$ in \eqref{eq:SemiC0A0HeatDir1NoDel} and using \eqref{eq:SemiWOdesHeatDirDel} we find that for $t\geq 0$ $w^{N_0}(t)$ and $w^{N-N_0}(t)$ satisfy
\begin{equation}\label{eq:SemiStateModsDel}
\begin{array}{lll}
&\dot{w}^{N_0}(t) = \tilde{A}_0w^{N_0}(t)+\tilde{B}_0v(t-r)+G^{N_0}(t),\\
&\dot{w}^{N-N_0}(t) = A_1w^{N-N_0}(t)+B_1v(t-r)+G^{N-N_0}(t).
\end{array}
\end{equation}

\subsection{Finite-dimensional observer-based controller design}
Consider the ODEs satisfied by $w^{N_0}(t)$, given in \eqref{eq:SemiStateModsDel}. In order to deal with the input delay $r>0$ therein, we fix $M\in \mathbb{N}$ and subdivide $r$ into $M$ parts of equal size $\frac{r}{M}$. We first consider $M\geq 2$ and design a chain of sub-predictors (observers of future state)
\begin{equation}\label{eq:SemiChainSubPred}
\begin{array}{lll}
&\hat{w}^{j}_1(t-r)\mapsto \cdots\mapsto\hat{w}^{j}_i\left(t-\frac{M-i+1}{M}r\right)\mapsto\cdots\\
&\mapsto\hat{w}^{j}_{M}\left(t-\frac{1}{M}r\right)\mapsto {w}^{j}(t), \quad j\in \left\{N_0, N-N_0\right\}.
\end{array}
\end{equation}
Here  $\hat{w}^{j}_i\left(t-\frac{M-i+1}{M}r\right)\mapsto\hat{w}^{j}_{i+1}\left(t-\frac{M-i}{M}r\right)$ means that $\hat{w}^{j}_i(t)$ predicts  the value of $\hat{w}^{j}_{i+1}(t+\frac{r}{M})$. Similarly, $\hat{w}^{j}_M(t)$ predicts the value of $w^{j}(t+\frac{r}{M})$.
\begin{remark}
Differently from the linear case \cite{katz2021sub}, here the sub-predictors are constructed for both $w^{N_0}(t)$ and $w^{N-N_0}(t)$. This is due to  semilinearity in \eqref{eq:SemilinearHeatNoDel}, which leads to coupling between all modes of the solution.
\end{remark}
\vspace{-0.1cm}
We assume the following:
\newline
\textbf{Assumption 2:} The point $x_*\in [0,1]$ satisfies \eqref{eq:SemiXstarassumpNoDel} and
	$\psi(x_*)\neq 0.$

Note that Assumption 2 holds for the particular case $x_*=0$ of non-collocated measurement. Recall the notations in \eqref{eq:SemiC0A0HeatDirNoDel} and let
\begin{equation}\label{eq:tildeC_0}
	\tilde{C}_0 = [\psi(x_*),C_0].
\end{equation}
 Under Assumption 2, the pair $(\tilde{A}_0,\tilde{C}_0)$ is observable by the Hautus lemma. Let $L_0 \in \mathbb{R}^{N_0+2}$ satisfy the Lyapunov inequality \eqref{eq:SemiGainsDesignLHeatDirNoDel} with $0<P_{\text{o}}\in \mathbb{R}^{(N_0+2)\times (N_0+2)}$ and $A_0,C_0$ replaced by $\tilde{A}_0,\tilde{C}_0$, respectively. We further choose the remaining gains as $l_n = 0, \ N_0+1\leq n\leq N$. Similarly, by the Hautus lemma, the pair $(\tilde{A}_0, \tilde{B}_0)$ is controllable. Let $K_0\in \mathbb{R}^{1\times (N_0+2)}$ satisfy \eqref{eq:SemiGainsDesignKHeatDirNoDel} with $0<P_{\text{c}}\in \mathbb{R}^{(N_0+2)\times (N_0+2)}$.

For $0\leq n \leq N$ and $1\leq i\leq M$ denote
\begin{equation}\label{eq:gi}
\begin{array}{lll}
 \hat{g}^{(i)}_n(t) =\\
  \left<g\left(t+\frac{(M+1-i)r}{M},\cdot \ ,Q(\cdot)\operatorname{col}\left\{\hat{w}^{N_0}_i(t),\hat{w}^{N-N_0}_i(t)\right\}\right), \phi_n \right>,\\
Q^T(x) = \operatorname{col}\left\{\psi(x),\phi_0(x),...,\phi_N(x) \right\},\\
\hat{G}^{N_0}_i(t) = \operatorname{col}\left\{0,\hat{g}^{(i)}_n(t)\right\}_{n=0}^{N_0},\\
\hat{G}^{N-N_0}_i(t) = \operatorname{col}\left\{\hat{g}^{(i)}_n(t)\right\}_{n=N_0+1}^{N}. 
\end{array}
\end{equation}
The sub-predictors satisfy the following ODEs for $t\geq 0$
\begin{equation}\label{eq:SemiSubpredODEsDel}
\begin{array}{lll}
&\dot{\hat{w}}^{N_0}_M(t) = \tilde{A}_0\hat{w}^{N_0}_M(t)+\tilde{B}_0v\left(t-\frac{M-1}{M}r\right)+\hat{G}^{N_0}_M(t)\\
&\hspace{5mm} -L_0\left[\tilde{C}_0\hat{w}^{N_0}_M(t-\frac{r}{M})+ C_1 \hat{w}^{N-N_0}_M(t-\frac{r}{M})-y(t)\right]\\
&\dot{\hat{w}}^{N-N_0}_M(t) = A_1\hat{w}^{N-N_0}_M(t)+B_1v\left(t-\frac{M-1}{M}r\right)\\
&\hspace{20mm}+\hat{G}^{N-N_0}_M(t),\\
&\dot{\hat{w}}^{N_0}_i(t) = \tilde{A}_0\hat{w}^{N_0}_i(t)+\tilde{B}_0v\left(t-\frac{i-1}{M}r\right)+\hat{G}^{N_0}_i(t)\\
&\hspace{5mm} -L_0\left[\tilde{C}_0\hat{w}^{N_0}_i(t-\frac{r}{M})+ C_1 \hat{w}^{N-N_0}_i(t-\frac{r}{M})\right.\\
&\hspace{13mm}\left.-\tilde{C}_0\hat{w}^{N_0}_{i+1}(t)- C_1 \hat{w}^{N-N_0}_{i+1}(t)\right],\\
&\dot{\hat{w}}^{N-N_0}_i(t) = A_1\hat{w}^{N-N_0}_i(t)+B_1v\left(t-\frac{i-1}{M}r\right)\\
&\hspace{20mm}+\hat{G}^{N-N_0}_i(t),\quad  1\leq i \leq M-1
\end{array}
\end{equation}
subject to
\begin{equation}\label{eq:SemiSubpredODEsDel00}
\hat{w}^{N_0}_i(t) = 0, \ \hat{w}^{N-N_0}_i(t) = 0, \ \ 1\leq i \leq M, \ t\leq 0.
\end{equation}
Note that as $i$ decreases, the input delay on the right-hand-side of the ODEs in \eqref{eq:SemiSubpredODEsDel} decreases by $\frac{r}{M}$. For the case $M=1$, the ODEs have the following form
\begin{equation}\label{eq:SemiSubpredODEs2Del}
\begin{array}{lll}
&\dot{\hat{w}}^{N_0}_1(t) = \tilde{A}_0\hat{w}^{N_0}_1(t)+\tilde{B}_0v\left(t-r\right)+\hat{G}^{N_0}_1(t)\\
&\hspace{5mm} -L_0\left[\tilde{C}_0\hat{w}^{N_0}_1(t-r)+ C_1 \hat{w}^{N-N_0}_1(t-r)-y(t)\right]\\
&\dot{\hat{w}}^{N-N_0}_1(t) = A_1\hat{w}^{N-N_0}_1(t)+B_1v\left(t-r\right)+\hat{G}^{N-N_0}_1(t).
\end{array}
\end{equation}

The finite-dimensional observer $\hat{w}(x,t)$ of the state $w(x,t)$, based on the $M\times (N+2)$ dimensional system of ODEs \eqref{eq:SemiSubpredODEsDel} is then given by
\begin{equation}\label{eq:SemiWhatSeries0SubDel}
\begin{array}{lll}
&\hat{w}(x,t)= \hat{w}^{N_0}_1(t-r)\cdot \text{col}\left\{0,\phi_n(x)\right\}_{n=0}^{N_0}\\
&\hspace{11mm}+\hat{w}_1^{N-N_0}(t-r)\cdot\text{col}\left\{\phi_{n}(x)\right\}_{n=N_0+1}^{N}.
\end{array}
\end{equation}
The controller is further chosen as
\begin{equation}\label{eq:SemiControllerSubDel}
v(t)=-K_0{\hat{w}}^{N_0}_1(t).
\end{equation}
In particular, \eqref{eq:SemiSubpredODEsDel} and \eqref{eq:SemiSubpredODEsDel00} imply continuity of $v(t)$ and $v(t)=0$ for $t\leq 0$.

Well-posedness of the closed-loop system \eqref{eq:SemiPDE1HeatNeuSampDel}, \eqref{eq:SemiSubpredODEsDel} subject to the control law \eqref{eq:SemiControllerSubDel} follows from arguments similar to
\eqref{eq:AbsStatNoDel}-\eqref{eq:xisolNoDel1} combined with the step method, meaning proof of well-posedness step by step on the intervals $[\frac{jr}{M},\frac{(j+1)r}{M}), \ j=0,1,\dots$ (see Section A of \cite{katz2021sub}, where such arguments have been used for sub-predictors). In particular, given $w(\cdot,0)\in H^1(0,1)$ we obtain a unique classical solution satisfying $w(\cdot,t)\in C\left([0,\infty);L^2(0,1)\right)\cap C^1\left( (0,\infty);L^2(0,1)\setminus \mathcal{J}\right)$ with $\mathcal{J} = \left\{\frac{jr}{M} \right\}_{j=0}^{\infty}$. Furthermore, $w(\cdot,t)\in \mathcal{D}\left(\mathcal{A} \right)$ for all $t>0$. We omit the details due to space constraints.

We define the estimation errors as follows
\begin{equation}\label{eq:SemiEstErrSubDel}
\begin{array}{lll}
&e^{N_0}_M(t) = w^{N_0}(t)-\hat{w}^{N_0}_M(t-\frac{r}{M}),\\
&e^{N-N_0}_M(t) = w^{N-N_0}(t)-\hat{w}^{N-N_0}_M(t-\frac{r}{M}), \\
&e^{N_0}_i(t) = \hat{w}_{i+1}^{N_0}(t-\frac{M-i}{M}r)-\hat{w}^{N_0}_i(t-\frac{M-i+1}{M}r),\\
&e^{N-N_0}_i(t) = \hat{w}_{i+1}^{N-N_0}(t-\frac{M-i}{M}r)\\
&\hspace{12mm}-\hat{w}^{N-N_0}_i(t-\frac{M-i+1}{M}r), \ 1\leq i \leq M-1.
\end{array}
\end{equation}
Then, the innovation term on the right-hand-side of the ODEs for $\hat{w}^{N_0}_M(t)$ given in \eqref{eq:SemiSubpredODEsDel} can be presented as
\begin{equation}\label{eq:SemiEstErrInnovDel}
\begin{array}{lll}
&\tilde{C}_0\hat{w}^{N_0}_M(t-\frac{r}{M})+ C_1 \hat{w}^{N-N_0}_M(t-\frac{r}{M})-y(t)\\
&\overset{\eqref{eq:SemiInDomPointMeas1HeatDirDel}}{=}-\tilde{C}_0 e^{N_0}_M(t) - C_1e^{N-N_0}_M(t) - \zeta(t).
\end{array}
\end{equation}
Here, $\zeta(t)$ is given in \eqref{eq:InnovNoDel} and satisfies the estimate \eqref{eq:ZetaEst} with $\Gamma>0$. Furthermore, by \eqref{eq:SemiEstErrSubDel}, we have
\begin{equation}\label{eq:SemiNoDelayControlPresentDel}
\begin{array}{l}
\hat w_1^{N_0}(t-r)+\sum_{i=1}^Me^{N_0}_i(t)=w^{N_0}(t).
\end{array}
\end{equation}
In particular, if the errors $e^{N_0}_i(t), \ 1\leq i \leq M$ converge to zero, we have $\hat{w}^{N_0}_1(t)\rightarrow {w}^{N_0}(t+r)$, meaning that $\hat{w}_1^{N_0}(t)$ predicts the future system state $w^{N_0}(t+r)$.

Using \eqref{eq:SemiStateModsDel}, \eqref{eq:SemiSubpredODEsDel} and \eqref{eq:SemiEstErrInnovDel} we obtain the following dynamics of the estimation errors for $t\geq 0$
\begin{equation}\label{eq:SemiEstErrSub1Del0}
\begin{array}{lll}
&\hspace{-2mm}\dot{e}^{N_0}_M(t) = \left(\tilde{A}_0-L_0\tilde{C}_0 \right)e^{N_0}_M(t)-L_0C_1e^{N-N_0}_M(t)+L_0\tilde{C}_0\\
&\hspace{-2mm}\times\Upsilon^{N_0}_{M,r}(t)+L_0C_1\Upsilon^{N-N_0}_{M,r}(t)-L_0\zeta(t-\frac{r}{M})+H^{N_0}_M(t)\\
&\hspace{-2mm}\dot{e}^{N-N_0}_M(t) =A_1e^{N-N_0}_M(t)+H^{N-N_0}_M(t),\\
&\hspace{-2mm}\dot{e}^{N_0}_{M-1}(t) = \left(\tilde{A}_0-L_0\tilde{C}_0 \right)e^{N_0}_{M-1}(t)-L_0C_1e^{N-N_0}_{M-1}(t)\\
&\hspace{-2mm}+L_0\tilde{C}_0\Upsilon^{N_0}_{M-1,r}(t)+L_0C_1\Upsilon^{N-N_0}_{M-1,r}(t)+L_0\tilde{C}_0e^{N_0}_M(t)\\
&\hspace{-2mm}-L_0\tilde{C}_0\Upsilon^{N_0}_{M,r}(t)+L_0C_1e^{N-N_0}_M(t)-L_0C_1\Upsilon^{N-N_0}_{M,r}(t) \\
&\hspace{-2mm}+L_0\zeta(t-\frac{r}{M})+H^{N_0}_{M-1}(t),\\
&\hspace{-2mm}\dot{e}^{N-N_0}_{M-1}(t) = A_1e^{N-N_0}_{M-1}(t)+H^{N-N_0}_{M-1}(t)
\end{array}
\end{equation}
whereas for $1\leq i \leq M-2$
\begin{equation}\label{eq:SemiEstErrSub1Del}
	\begin{array}{lll}
		&\hspace{-2mm}\dot{e}^{N_0}_i(t) = (\tilde{A}_0-L_0\tilde{C}_0)e^{N_0}_i(t)-L_0C_1e^{N-N_0}_i(t)\\
		&\hspace{-2mm}+L_0\tilde{C}_0e^{N_0}_{i+1}(t)+L_0C_1e^{N-N_0}_{i+1}(t)+L_0\tilde{C}_0\Upsilon^{N_0}_{i,r}(t)\\
		&\hspace{-2mm}+L_0C_1\Upsilon^{N-N_0}_{i,r}(t)-L_0\tilde{C}_0\Upsilon^{N_0}_{i+1,r}(t)\\
		&\hspace{-2mm}-L_0C_1\Upsilon^{N-N_0}_{i+1,r}(t)+H^{N_0}_i(t),\\
		&\hspace{-2mm}\dot{e}^{N-N_0}_i(t) =A_1e^{N-N_0}_i(t)+H^{N-N_0}_i(t) .
	\end{array}
\end{equation}
Here
\begin{equation}\label{eq:SemiEstErrSub2Del}
\begin{array}{lll}
&\hspace{-2mm}\Upsilon^{N_0}_{i,r}(t) = e^{N_0}_i(t)-e^{N_0}_i(t-\frac{r}{M}),\\
&\hspace{-2mm}\Upsilon^{N-N_0}_{i,r}(t) = e^{N-N_0}_i(t)-e^{N-N_0}_i(t-\frac{r}{M}),\\
&\hspace{-2mm}H^{N_0}_M(t) = G^{N_0}(t)-\hat{G}^{N_0}_M(t-\frac{r}{M}),\\
&\hspace{-2mm}H^{N-N_0}_M(t) = G^{N-N_0}(t)-\hat{G}^{N-N_0}_M(t-\frac{r}{M}),\\
&\hspace{-2mm}H^{N_0}_i(t) = \hat{G}^{N_0}_{i+1}(t-\frac{M-i}{M}r)-\hat{G}^{N_0}_i(t-\frac{M-i+1}{M}r),\\
&\hspace{-2mm}H^{N-N_0}_i(t) = \hat{G}^{N-N_0}_{i+1}(t-\frac{M-i}{M}r)-\hat{G}^{N-N_0}_i(t-\frac{M-i+1}{M}r).
\end{array}
\end{equation}
From \eqref{eq:SemiStateModsDel}, \eqref{eq:SemiControllerSubDel} and \eqref{eq:SemiNoDelayControlPresentDel} we further have
\begin{equation}\label{eq:SemiStateMods1Del}
\begin{array}{lll}
&\dot{w}^{N_0}(t) = \left(\tilde{A}_0-\tilde{B}_0K_0 \right)w^{N_0}(t)+\tilde{B}_0K_0\sum_{i=1}^Me^{N_0}_i(t)\\
&\hspace{15mm}+G^{N_0}(t),\\
&\dot{w}^{N-N_0}(t) = A_1w^{N-N_0}(t)+B_1K_0\sum_{i=1}^Me^{N_0}_i(t)\\
&\hspace{16mm}+G^{N-N_0}(t).
\end{array}
\end{equation}
We introduce the notations
\begin{equation}\label{eq:SeminotationsDelayMSubNewDel0}
	\begin{array}{lll}
		&X(t) = \operatorname{col}\left\{w^{N_0}(t),w^{N-N_0}(t)\right\},\\
		&X_e(t) = \text{col}\left\{e^{N_0}_1(t),e^{N-N_0}_1(t),\dots,e^{N_0}_M(t),e^{N-N_0}_M(t) \right\},\\
& \Upsilon_{e,r}(t) = X_e(t)-X_e\left(t-\frac{r}{M}\right),\\
	&		H(t) = \text{col}\left\{H^{N_0}_1(t),H^{N-N_0}_1(t),\dots,H^{N_0}_M(t),H^{N-N_0}_M(t) \right\}
	\end{array}
\end{equation}
and
\begin{equation}\label{eq:SeminotationsDelayMSubNewDel}
\begin{array}{lll}
&G(t) = \operatorname{col}\left\{G^{N_0}(t), G^{N-N_0}(t)\right\},\\
	&F_X = \scriptsize\begin{bmatrix}
		\tilde{A}_0-\tilde{B}_0K_0 & 0 \\
		-B_1K_0 & A_1
	\end{bmatrix}, \ B_X = \operatorname{col}\left\{\tilde{B}_0, B_1\right\},\\
	&\mathcal{I} = [I_{N_0+2} \ 0 \ I_{N_0+2} \ 0 \dots I_{N_0+2} \ 0]\in \mathbb{R}^{1\times M(N+2)},\\
&F_0 = \scriptsize\begin{bmatrix}
		\tilde{A}_0-L_0\tilde{C}_0 & -L_0C_1\\
		0 & A_1
\end{bmatrix}, \ \mathcal{L}_0 = \begin{bmatrix}
		L_0 \\ 0
\end{bmatrix}, \ \mathcal{C} = [\tilde{C}_0 \ C_1],\\
&F_e = I_M \otimes F_0 + J_{0,M} \otimes \mathcal{L}_0\mathcal{C} ,\ \tilde{K}_0 = [K_0, 0_{1\times(N-N_0)}]\\
&\Lambda_e = I_M \otimes \mathcal{L}_0\mathcal{C} - J_{0,M} \otimes \mathcal{L}_0\mathcal{C} ,\\
&\mathcal{L}_{\zeta} = \operatorname{col} \left\{0, 0, \dots, 0, \mathcal{L}_0, -\mathcal{L}_0\right\}\in \mathbb{R}^{ M(N+2)}.
\end{array}
\end{equation}
Here $J_{0,M}$ is an upper triangular Jordan block of order $M$ with zero diagonal and $\otimes$ is the Kronecker product. Then, from \eqref{eq:SemiWOdesHeatDirDel}, \eqref{eq:SemiEstErrSub1Del}, \eqref{eq:SemiStateMods1Del} and \eqref{eq:SeminotationsDelayMSubNewDel} we obtain the following closed-loop system for $t\geq 0$
\begin{equation}\label{eq:SemiClosedLoopDel}
\begin{array}{lll}
&\dot{X}(t) =F_XX(t)+B_XK_0\mathcal{I}X_e(t)+G(t) ,\\
&\dot{X}_e(t) = F_eX_e(t)+\Lambda_e\Upsilon_{e,r}(t)+\mathcal{L}_{\zeta}\zeta(t-\frac{r}{M})+H(t),\\
&\dot{w}_n(t) = -\lambda_nw_n(t) + g_n(t)- b_n \tilde{K}_0X(t)\\
&\hspace{13mm}+b_nK_0 \mathcal{I}X_e(t) , \quad n>N.
\end{array}
\end{equation}
Differently from the existing finite-dimensional controllers \cite{RamiContructiveFiniteDim,katz2020constructiveDelay}, where the closed-loop systems is written in terms of the observer and the tail $w_n(t) \ (n>N)$,
here 
\eqref{eq:SemiClosedLoopDel} is presented in terms of the state $X(t)$, the estimation errors $X_e(t)$ and the tail. This allows to eliminate the delay $r$ from the ODEs of $X(t)$ and $w_n(t), \ n>N$ while decreasing it to  $\frac{r}{M}$ (which is small for large $M$) in the ODEs of $X_e(t)$.

\subsection{$H^1$- stability analysis}
For $H^1$-stability analysis of \eqref{eq:SemiClosedLoopDel} we define the Lyapunov functional
\begin{equation}\label{eq:SemiVComponents0NewDel}
\begin{array}{lll}
&V(t):=V_X(t)+V_{e}(t)+ V_{q}(t),\\
&V_X(t)=\left|X(t) \right|^2_{P_X}+\sum_{n=N+1}^{\infty}\lambda_nw^2_n(t),\\
&V_{q}(t)=q\int_{t-{r\over M}}^{t} e^{-2\delta(t-s)}\zeta^2(s) ds,\\
&V_e(t) = \left|X_e(t) \right|^2_{P_e}+ V_{S_e}(t)+V_{R_e}(t)
\end{array}
\end{equation}
Here $0<P_X$ and $0<P_e$ are matrices of appropriate dimensions, whereas $0<q$ is a scalar. Furthermore, $V_{S_e}(t)$ and $V_{R_e}(t)$ are given by
\begin{equation}\label{eq:VComponentsRNewDel}
\begin{array}{llll}
&V_{S_e}(t):=\int_{t-{r\over M}}^{t} e^{-2\delta(t-s)}\left| X_e(s)\right|_{S_e}^2 ds,\\
&V_{R_e}(t):= {r\over M} \int_{-{r\over M}}^{0} \int_{t+\theta}^t e^{-2\delta(t-s)} \left|\dot{X}_e(s)\right|^2_{R_e} ds d\theta
\end{array}
\end{equation}
where $0<S_e$ and $0<R_e$ are matrices of appropriate dimension. Note that $V_X(t)$ allows to compensate $\zeta(t)$ using \eqref{eq:ZetaEst}, $V_q(t)$ compensates $\zeta(t-\frac{r}{M})$ , whereas $V_e(t)$ compensate the delay $\frac{r}{M}$ appearing in the ODEs of $X_e(t)$.

Differentiation of $V_q(t)$ gives
\begin{equation}\label{eq:SemiVQDiffNewDel}
\dot{V}_q+2\delta V_q = q\zeta^2(t)-q\varepsilon_{r,M}\zeta^2\left(t-\frac{r}{M}\right), \ \varepsilon_{r,M} = e^{- \frac{2\delta r}{M}}.
\end{equation}
Differentiating $V_X(t)$ along the solution to \eqref{eq:SemiVComponents0NewDel} we obtain
\begin{equation}\label{eq:SemiVnomDelayNew}
\begin{array}{lll}
&\hspace{-2mm}\dot{V}_X+2\delta V_X =  X^T(t)\left[P_XF_X +F_X^TP_X+2\delta P_X \right]X(t)\\
&\hspace{-2mm}+2X^T(t)P_XB_XK_0\mathcal{I}X_e(t)+2X^T(t)P_XG(t)\\
&\hspace{-2mm}+2\sum_{n=N+1}^{\infty}\left(-\lambda_n^2+\delta \lambda_n \right)w_n^2(t)\\
&\hspace{-2mm}+2\sum_{n=N+1}^{\infty}\lambda_nw_n(t)\left[g_n(t)-b_n\left(\tilde{K}_0X(t)-K_0\mathcal{I}X_e(t)\right)\right].
\end{array}
\end{equation}
Let $\alpha_1>0$. By the Young inequality we have
\begin{equation}\label{eq:SemiCrosTermDel}
\begin{array}{lll}
&2\sum_{n=N+1}^{\infty}\lambda_nw_n(t)g_n(t)\\
&\leq \frac{1}{\alpha_1}\sum_{n=N+1}^{\infty}\lambda_n^2w_n^2(t) - \alpha_1\left|G(t)\right|^2 + \alpha_1\sum_{n=0}^{\infty}g_n^2(t).
\end{array}
\end{equation}
By Parseval's equality and the assumptions on $g$ we have
\begin{equation}\label{eq:SemiCrosTerm1Del}
\begin{array}{lll}
&\sum_{n=0}^{\infty}g_n^2(t) \overset{\eqref{eq:gnDefDel}}{=}\int_0^1\left|g\left(t,s,w(s,t)+\psi(s)u(t-r) \right)\right|^2ds\\
&\overset{\eqref{eq:NonLinHeatVectfNoDel}}{\leq} \sigma^2\int_0^1\left[w(s,t)+\psi(s)u(t-r) \right]^2ds\\
&\leq 2\sigma^2\int_0^1w^2(s,t)ds + 2u^2(t-r)\sigma^2\int_0^1\psi^2(s)ds\\
& = 2\sigma^2X^T(t)\Xi_X X(t)+2\sigma^2\sum_{n=N+1}^{\infty}w_n^2(t), \\
&\Xi_X \overset{\eqref{eq:PsiProp}}{=} \operatorname{diag}\left\{\frac{2}{\pi^2},I_{N+1}\right\}.
\end{array}
\end{equation}
Similarly, we have for $\alpha_2,\alpha_3>0$
\begin{equation}\label{eq:SemiCrosTerm2Del}
\begin{array}{lll}
&-2\sum_{n=N+1}^{\infty}\lambda_nw_n(t)b_n\tilde{K}_0X(t)\\
&\overset{\eqref{eq:bSeriesNoDel}}{\leq} \frac{1}{\alpha_2}\sum_{n=N+1}^{\infty}\lambda_nw_n^2(t)+\frac{2\alpha_2\xi_{N+1}}{\pi^2}\left|\tilde{K}_0X(t) \right|^2,\\
&2\sum_{n=N+1}^{\infty}\lambda_nw_n(t)b_nK_0\mathcal{I}X_e(t)\\
&\overset{\eqref{eq:bSeriesNoDel}}{\leq} \frac{1}{\alpha_3}\sum_{n=N+1}^{\infty}\lambda_nw_n^2(t)+\frac{2\alpha_3\xi_{N+1}}{\pi^2}\left|K_0\mathcal{I}X_e(t)\right|^2
\end{array}
\end{equation}

Differentiation of  $V_{e}(t)$ and Jensen's inequality lead to
\begin{equation}\label{eq:SemiCrossTerms4Del}
\begin{array}{lll}
&\hspace{-5mm}\dot{V}_{e}+2\delta V_{e} \leq X_e^T(t)\left[P_eF_e+F_e^TP_e+2\delta P_e \right]X_e(t)\\
&\hspace{-5mm}+2X_e^T(t)P_e\Lambda_e\Upsilon_{e,r}(t)+2X_e^T(t)P_e\mathcal{L}_{\zeta}\zeta(t-\frac{r}{M})\\
&\hspace{-5mm}+2X_e^T(t)P_eH(t)+\left| X_e(t)\right|_{S_e}^2-\varepsilon_{r,M}\times\\
&\hspace{-5mm}\left[\left|X_e(t)-\Upsilon_{e,r}(t) \right|_{S_e}^2+\left|\Upsilon_{e,r}(t) \right|_{R_e}^2\right]+\left(\frac{r}{M}\right)^2\left|\dot{X}_e(t) \right|_{R_e}^2.
\end{array}
\end{equation}

Recall $G^{N_0}(t), G^{N-N_0}(t)$ in \eqref{eq:SemiC0A0HeatDir1Del},  $\left\{\hat{G}_i^{N_0}(t), \hat{G}_i^{N-N_0}(t)\right\}_{i=1}^M$ in \eqref{eq:gi}, the estimation errors in \eqref{eq:SemiEstErrSubDel} and $H(t)$ defined in \eqref{eq:SemiEstErrSub2Del} and \eqref{eq:SeminotationsDelayMSubNewDel0}.  By Parseval's equality we have
\begin{equation*}\label{eq:SemiHestimateDel0}
	\begin{array}{lll}
&\left|H_M^{N_0}(t)\right|^2+\left|H_M^{N-N_0}(t)\right|^2 \\
&= \sum_{n=0}^{N}\left[g_n(t)-\hat{g}_n^{(M)}(t-\frac{r}{M})\right]^2\\
&\overset{\eqref{eq:SemiC0A0HeatDir1Del},\eqref{eq:gi}}{=}\int_0^1\left|g\left(t,s,w(s,t)+\psi(s)u(t-r)\right)\right.
	\end{array}
\end{equation*}
\begin{equation}\label{eq:SemiHestimateDel}
\begin{array}{lll}
& \left. -g\left(t,s,Q_1(s)\hat{w}^{N_0}_M(t-\frac{r}{M})+Q_2(s)\hat{w}^{N-N_0}_M(t-\frac{r}{M})\right) \right|^2ds\\
&\overset{\eqref{eq:NonLinHeatVectfNoDel}}{\leq}\sigma^2 \int_0^1\left[w(s,t)+\psi(s)u(t-r)-Q_1(s)\hat{w}^{N_0}_M(t-\frac{r}{M}) \right.\\
&\hspace{15mm}\left. - Q_2(s)\hat{w}^{N-N_0}_M(t-\frac{r}{M})\right]^2ds\\
&\leq 2\sigma^2e^{N_0,T}_M(t)\Xi_1e^{N_0}_M(t)+2\sigma^2\left|e^{N-N_0}_M(t) \right|^2\\
&\hspace{3mm}+2\sigma^2\sum_{n=N+1}^{\infty}w_n^2(t),\\
&\left|H_i^{N_0}(t)\right|^2+\left|H_i^{N-N_0}(t)\right|^2 \leq 2\sigma^2e^{N_0,T}_i(t)\Xi_1e^{N_0}_i(t)\\
&+2\sigma^2\left|e^{N-N_0}_i(t) \right|^2,\\
&\Xi_1 \overset{\eqref{eq:PsiProp}}{=} \left\{\frac{2}{\pi^2},I_{N_0+1}\right\}, \quad 1\leq i \leq M-1
\end{array}
\end{equation}
By \eqref{eq:SeminotationsDelayMSubNewDel0} and \eqref{eq:SeminotationsDelayMSubNewDel}, the latter implies
\begin{equation}\label{eq:SemiHestimate1Del}
\begin{array}{lll}
&\left|H(t)\right|^2 \leq 2\sigma^2 X_e^T(t)\Xi_EX_e(t) + 2\sigma^2 \sum_{n=N+1}^{\infty}w_n^2(t),\\
&\Xi_E = \operatorname{diag}\left\{\Xi_1,I_{N-N_0},\dots, \Xi_1,I_{N-N_0}\right\}.
\end{array}
\end{equation}
Let $\eta(t) = \operatorname{col}\left\{X(t),G(t),X_e(t),\zeta(t-\frac{r}{M}),\Upsilon_{e,r}(t),H(t) \right\}$. By \eqref{eq:SemiVQDiffNewDel}-\eqref{eq:SemiHestimate1Del} and the S-procedure, we have for $\beta>0$
\begin{equation}\label{eq:VestDel}
\begin{array}{lll}
&\dot{V}+2\delta V+\beta \left\{2\sigma^2 X_e^T(t)\Xi_EX_e(t)  \right.\\
&\hspace{5mm} \left.+2\sigma^2 \sum_{n=N+1}^{\infty}w_n^2(t)-\left|H(t)\right|^2\right\}\\
&\leq \eta^T(t)\Psi_1\eta(t) + q\zeta^2(t)+2\sum_{n=N+1}^{\infty}\varpi_nw_n^2(t)
\end{array}
\end{equation}
where
\begin{equation*}\label{eq:Vest2Del}
\begin{array}{lll}
&\varpi_n =\left(-1+\frac{1}{2\alpha_1}\right)\lambda_n^2+\left(\delta+\frac{1}{2\alpha_2}+\frac{1}{2\alpha_3}\right)\lambda_n\\
&\hspace{15mm}+\sigma^2\left(\alpha_1+\beta \right) , \ n>N,\\
		&\Psi_1 = \scriptsize \left[
		\begin{array}{c|c|c}
			\Phi_1 & \begin{matrix} P_XB_XK_0\mathcal{I} \ & 0  \ & 0\\
				0 \ & 0 \ & 0
			\end{matrix} &  \begin{matrix}
				0 \\  0
			\end{matrix}\\
			\hline
			* & \Phi_2 & \begin{matrix}
				P_e \\ 0 \\ 0
			\end{matrix}\\
			\hline
			* & * & -\beta I
		\end{array}
		\right]+\left(\frac{r}{M} \right)^2\Theta^TR_e\Theta,\\
		&\Phi_1 = \scriptsize \begin{bmatrix}
			\varphi_1 & P_X \\* & -\alpha_1 I
		\end{bmatrix}, \ \Phi_2 = \begin{bmatrix}
			\varphi_2 & P_e \mathcal{L}_{\zeta} & P_e \Lambda_e - \varepsilon_{r,M}S_e\\
			* & -q\varepsilon_{r,M} & 0\\
			* & * &-\varepsilon_{r,M}(S_e+R_e)
		\end{bmatrix},\\
		&\Theta = [0 , 0, F_e, \mathcal{L}_{\zeta}, \Lambda_e , I],\\
&\varphi_1 = P_XF_X +F_X^TP_X+2\delta P_X\\
&\hspace{8mm}+2\alpha_1\sigma^2\Xi_X +\frac{2\alpha_2\xi_{N+1}}{\pi^2}\tilde{K}_0^T\tilde{K}_0\\
&\varphi_2 = P_eF_e +F_e^TP_e+2\delta P_e+\frac{2\alpha_3\xi_{N+1}}{\pi^2}\mathcal{I}^TK_0^TK_0\mathcal{I}\\
&\hspace{8mm}+2\beta \sigma^2\Xi_E+\left(1-\varepsilon_{r,M} \right)S_e.
\end{array}
\end{equation*}
To compensate $\zeta^2(t)$ in \eqref{eq:VestDel} we use \eqref{eq:ZetaEst} and  monotonicity of $\left\{\lambda_n\right\}_{n=1}^{\infty}$ as follows
\begin{equation}\label{eq:Vest4Del}
\begin{array}{lll}
&q\zeta^2(t)+2\sum_{n=N+1}^{\infty}\varpi_nw_n^2(t)\\
&\overset{\eqref{eq:ZetaEst}}{\leq}\sum_{n=N+1}^{\infty}\left(2\varpi_{n}+q\kappa_{n}\right)w_n^2(t)\leq 0
\end{array}
\end{equation}
provided $\varpi_{N+1} +\frac{q\kappa_{N+1}}{2}<0$.
From \eqref{eq:VestDel} - \eqref{eq:Vest4Del} we have
\begin{equation}\label{eq:Vest5}
\begin{array}{lll}
&\dot{V}+2\delta V+\beta \left\{2\sigma^2 X_e^T(t)\Xi_2X_e(t)  \right.\\
&\hspace{5mm} \left.+2\sigma^2 \sum_{n=N+1}^{\infty}w_n^2(t)-\left|H(t)\right|^2\right\}\leq 0
\end{array}
\end{equation}
if
\begin{equation}\label{eq:Vest6}
\begin{array}{lll}
&\Psi_1<0, \quad \varpi_{N+1}+\frac{q\kappa_{N+1}}{2}<0.
\end{array}
\end{equation}
By Schur complement, we have that $\varpi_{N+1}+\frac{q\kappa_{N+1}}{2}<0$ iff
\begin{equation}\label{eq:Vest7Del}
\begin{array}{lll}
&\scriptsize \left[
\begin{array}{c|c}
\varphi_3 & \begin{matrix}
1 \qquad & 1 \qquad  & 1
\end{matrix}\\
\hline
* & -\frac{2}{\lambda_{N+1}}\operatorname{diag} \left\{\frac{\alpha_1}{\lambda_{N+1}},\alpha_2,\alpha_3\right\}
\end{array}
\right]<0,\\
&\varphi_3 = -\lambda^2_{N+1}+\left(\delta+\frac{q\Gamma}{2} \right)\lambda_{N+1}\\
&\hspace{15mm}+\sigma^2\left(\alpha_1+\beta \right)+\frac{q}{2}\left(1+\Gamma \right).
\end{array}
\end{equation}
Summarizing, we arrive at
\begin{theorem}\label{Thm:SemilinearDel}
Consider the system \eqref{eq:SemiPDE1HeatNeuSampDel} with boundary conditions \eqref{eq:SemiBCsHeatDir1}, point measurement \eqref{eq:SemiInDomPointMeas1HeatDirDel} and control law \eqref{eq:SemiControllerSubDel}. Assume that $g(t,x,z)$ is a locally Lipschitz function satisfying $g(t,x,0)\equiv 0$ and \eqref{eq:NonLinHeatVectfNoDel} for a given $\sigma > 0$. Let $\delta>0$, $N_0\in \mathbb{N}$ satisfy \eqref{eq:SemiN0HeatDirNoDel} and $N\in \mathbb{N}$ satisfy $N_0\leq N$. Let $L_0$ and $K_0$ be obtained using \eqref{eq:SemiGainsDesignLHeatDirNoDel} (with $A_0,C_0$ replaced by $\tilde{A}_0,\tilde{C}_0$) 
and \eqref{eq:SemiGainsDesignKHeatDirNoDel}, respectively. Given $M\in \mathbb{N}$ and $\Gamma>0$, let there exist positive definite matrices $P_X,P_e,S_e,R_e$ and scalars $q,\alpha_1,\alpha_2,\alpha_3,\beta>0$ such that \eqref{eq:Vest6} hold. Then, given $w(\cdot,0)\in H^1(0,1)$, the solution $u(t-r),w(x,t)$ of \eqref{eq:SemiPDE1HeatNeuSampDel} subject to the control law \eqref{eq:SemiControllerSubDel} and the observer $\hat{w}(x,t)$,  defined by \eqref{eq:SemiSubpredODEsDel} (with notations \eqref{eq:gi})
and \eqref{eq:SemiWhatSeries0SubDel}, satisfy
\begin{equation}\label{eq:StabDel}
\begin{array}{lll}
&u^2(t-r)+\left\|w(\cdot,t) \right\|_{H^1}^2\\
&\hspace{15mm}+ \left\|\hat{w}(\cdot,t) \right\|_{H^1}^2\leq D e^{-2\delta t}\left\|w(\cdot,0) \right\|_{H^1}^2
\end{array}
\end{equation}
for $ t\geq 0$ and some $D\geq 1$. Given $r>0$, \eqref{eq:Vest6} are always feasible for $M,N$ large enough and $\sigma>0$ small enough.
\end{theorem}
\begin{proof}
The upper bound \eqref{eq:StabDel} follows from arguments similar to \eqref{eq:Comp1NoDel1} and \eqref{eq:Comp1NoDel2} in Theorem \ref{Thm:SemilinearNoDel}. Next, we fix $r>0$ and treat feasibility of \eqref{eq:Vest6} for $M,N$ large enough and $\sigma>0$ small enough. For $\sigma=0$ (i.e. when $g\equiv 0$ in \eqref{eq:SemilinearHeatNoDel}), feasibility for large enough $M$ and $N$ follows from Theorem 1 in \cite{katz2021sub}. Fixing such $M$ and $N$ and using continuity of eigenvalues, we have that \eqref{eq:Vest6} are feasible provided $\sigma>0$ is small enough.
\end{proof}

\section{Numerical example}\label{sec_Examples}
Consider first \eqref{eq:SemilinearHeatNoDel} under Neumann actuation \eqref{eq:SemiBCsHeatDir} and boundary measurement \eqref{eq:SemiInDomPointMeasHeat}, where $x_*=0$. Recall that $g(t,x,z)$ is a locally Lipschitz function satisfying $g(t,x,0)\equiv 0$ and \eqref{eq:NonLinHeatVectfNoDel} for a given $\sigma > 0$. Let $\delta = 0.001$ be the desired decay rate and $N_0 = 0$. Let the gains $L_0$ and $K_0$ satisfy \eqref{eq:SemiGainsDesignLHeatDirNoDel} and \eqref{eq:SemiGainsDesignKHeatDirNoDel}, respectively. The gains are given by
\begin{equation*}
L_0 = 2.75, \ K_0 = \begin{bmatrix}-5.468 & 32.19 \end{bmatrix}.
\end{equation*}
Given $N\in \left\{4,5,\dots,9 \right\}$, the LMI of Theorem \ref{Thm:SemilinearNoDel} was verified using Matlab to obtain the largest value of $\sigma$ which preserves feasibility of the LMI. The results are presented in Table \ref{Tab:NoPred}.
\begin{table}
	\begin{center}
		\scalebox{0.9}{
			\begin{tabular}{|c|c|c|c|c|c|c|c|c|}
				\hline
				N & 3 & 4 & 5 &6 & 7 & 8 \\
				\hline
				$\sigma_{\text{max}}$ &0.39 & 0.47 & 0.59 & 0.64 & 0.76 & 0.83\\
				\hline
		\end{tabular}}
	\end{center}
	\caption{\label{Tab:NoPred} Theorem \ref{Thm:SemilinearNoDel}: Feasibility of LMI }
\end{table}
Next, consider \eqref{eq:SemilinearHeatNoDel} under Neumann actuation with constant input delay \eqref{eq:SemiBCsHeatDir} and boundary measurement \eqref{eq:SemiInDomPointMeasHeat}, where $x_*=0$. Let $\delta = 0.001$ be the desired decay rate, $\sigma = 0.5$ and $N_0 = 0$. Let the gains $L_0$ and $K_0$ be obtained using \eqref{eq:SemiGainsDesignLHeatDirNoDel} (with $C_0$ replaced by $\tilde{C}_0$ in \eqref{eq:tildeC_0}) and \eqref{eq:SemiGainsDesignKHeatDirNoDel}, respectively. The gains are given by
\begin{equation}\label{eq:GainsSim}
L_0 = \begin{bmatrix}7.33 &	1.01 \end{bmatrix}^T, \ K_0 = \begin{bmatrix}1.95 &	0.55 \end{bmatrix}.
\end{equation}
Given $M=2$ and $N\in \left\{4,5,6\right\}$, the LMIs of Theorem \ref{Thm:SemilinearDel} were verified to obtain the largest value of the input delay $r>0$ which preserves feasibility of the LMIs. The results are presented in Table \ref{Tab:SubPred}.
\begin{table}
	\begin{center}
		\scalebox{0.9}{
			\begin{tabular}{|c|c|c|c|}
				\hline
				N & 4 & 5 & 6  \\
				\hline
				$r_{\text{max}}$ &0.32 & 0.45 & 0.56 \\
				\hline
		\end{tabular}}
	\end{center}
	\caption{\label{Tab:SubPred} Theorem \ref{Thm:SemilinearDel}: Feasibility of LMIs ($\sigma = 0.5, \ M=2$) }
\end{table}

For simulations of the closed-loop system, consider \eqref{eq:SemilinearHeatNoDel} under Neumann actuation with constant input delay \eqref{eq:SemiBCsHeatDir}, boundary measurement \eqref{eq:SemiInDomPointMeasHeat} at $x_*=0$ and
\begin{equation*}
g(t,x,z) = \sigma\sin(t+3x+z).
\end{equation*}
We fix $\sigma = 0.5$, delay $r=0.32$, $N=4$ and $M=2$ subpredictors. Let the gains be given by \eqref{eq:GainsSim}. The ODE-PDE system \eqref{eq:SemiPDE1HeatNeuSampDel} and subpredictor ODEs \eqref{eq:SemiSubpredODEsDel} were simulated using the FTCS (Forward Time Centered Space) and Forward Euler finite-difference schemes, where the initial condition was chosen as
\begin{equation*}
w(x,0)= 8.5x(1-x), \quad x\in[0,1].
\end{equation*}
The simulation results are given in Figure \ref{fig:Sim} and confirm our theoretical analysis. Stability of the closed-loop system in simulation was preserved for $r = 0.63$, which implies that our approach is slightly conservative in this example.
\begin{figure}
	\centering
	\includegraphics[width=70mm,scale=0.11]{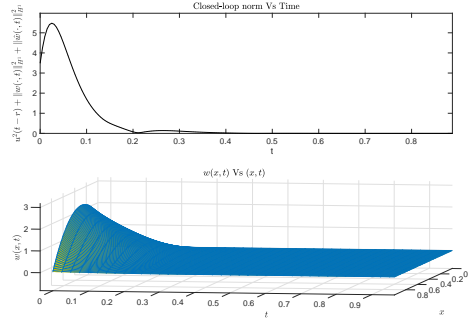}
	\caption{Closed-loop system simulation}\label{fig:Sim}
\end{figure}

\section{Conclusions}\label{sec_Concl}
In this paper we studied global boundary stabilization of a semilinear heat equation under point measurement. For the non-delayed case, we suggested a finite-dimensional nonlinear observer-based controller.
To compensate a constant input delay, we  constructed nonlinear sequential sub-predictors. A numerical example demonstrated the efficiency of the approach.
Our method in the future can be extended to other semilinear PDEs. 

\bibliographystyle{abbrv}
\bibliography{Bibliography250721}

\end{document}